\documentclass[a4paper, reqno]{amsart}

\usepackage{latexsym}
\usepackage{amsmath}
\usepackage{amsfonts}
\usepackage{amssymb}
\usepackage{amsxtra}
\usepackage{mathrsfs}
\usepackage{eucal}
\usepackage{tikz} \usetikzlibrary{calc} \tikzset{>=latex} \usetikzlibrary{backgrounds}
\usepackage[all]{xy}
\usepackage{color}
\usepackage{braket}
\usepackage{graphics, setspace}
\usepackage{etoolbox, textcomp}
\usepackage{comment}
\usepackage{changepage}
\usepackage{xcolor}
\definecolor{rouge}{rgb}{0.85,0.1,.4}
\definecolor{bleu}{rgb}{0.1,0.2,0.9}
\definecolor{violet}{rgb}{0.7,0,0.8}
\usepackage[colorlinks=true,linkcolor=bleu,urlcolor=violet,citecolor=rouge]{hyperref}
\usepackage{cleveref}
\usepackage{lscape}
\usepackage{subcaption}
\usepackage{hyperref}

\RequirePackage[numbers,sort&compress]{natbib}

\newtheorem{theorem}{Theorem}[section]
\newtheorem{lemma}[theorem]{Lemma}
\newtheorem{definition}[theorem]{Definition}
\newtheorem{example}[theorem]{Example}
\newtheorem{proposition}[theorem]{Proposition}
\newtheorem{corollary}[theorem]{Corollary}
\newtheorem{conjecture}[theorem]{Conjecture}
\newtheorem{remark}[theorem]{Remark}
\newtheorem{notation}[theorem]{Notation}
\def\CC{\mathbb{C}}

\def\ZZ{\mathbb{Z}}

\newcommand\cD{\mathcal{D}}

\newcommand\cF{\mathcal{F}}

\newcommand{\ch}{\textup{char}}

\newcommand{\Ext}{\textup{Ext}}

\newcommand\quash[1]{}
\newcommand\one{\mathbf{1}}

\renewcommand\a\alpha
\renewcommand\b\beta
\newcommand\ga\gamma
\renewcommand\d\delta
\newcommand\D\Delta

\newcommand{\slirr}[1]{\Mod{L}_{#1}}              % affine sl(2) irreducible modules with finite-dim top spaces
\newcommand{\sldis}[1]{\Mod{D}_{#1}}              % affine sl(2) irreducible modules with infinite-dim top spaces
\newcommand{\slindrel}[1]{\Mod{E}_{#1}}           % affine sl(2) indecomposable relaxed modules
\newcommand{\slrel}[2]{\slindrel{#1;\Delta_{#2}}} % affine sl(2) irreduciblerelaxed modules
\newcommand{\slproj}[1]{\Mod{P}_{#1}}             % affine sl(2) projective modules

\newcommand{\alg}[1]{\mathfrak{#1}}  % Lie algebras
\newcommand{\Mod}[1]{\mathcal{#1}}   % modules
\newcommand{\lat}[1]{\mathsf{#1}}    % lattices
\newcommand{\voa}{vertex operator algebra}

\newcommand{\sfaut}{\sigma}                   % spectral flow
\newcommand{\sfmod}[2]{\sfaut^{#1}(#2)}       % apply spectral flow #1 times to #2

\newcommand{\rlat}{\lat{Q}} % root lattice

\newcommand{\lra}{\longrightarrow}
\newcommand{\dses}[3]{0 \lra #1 \lra #2 \lra #3 \lra 0} % displayed ses

\def \<{\langle}
\def \>{\rangle}

\def \a{\alpha }

\def \ga{\gamma }

\def \b{\beta }

\def \ch{\text{ch}}

\newcommand{\bea}{\begin{eqnarray}}
\newcommand{\eea}{\end{eqnarray}}
\newcommand{\be}{\begin {equation}}
\newcommand{\ee}{\end{equation}}
\newcommand{\g}{\frak g}

\newcommand{\sltwo}{\mathfrak{sl}_2}
\newcommand{\asltwo}{\widehat{\mathfrak{sl}}_2}

\newcommand{\h}{\frak h}

\newcommand{\W}{\mathcal W}
\newcommand{\C}{\Bbb C}

\newcommand{\vak}{\bf 1}

\newcommand{\mc}{\mathcal}
\newcommand{\mf}{\mathfrak}
\newcommand{\on}{\operatorname}

\newcommand{\sWt}[1]{L_{#1}(\g)\on{-wtmod}}

\newcommand{\sWtsl}[1]{L_{#1}(\mf{sl}_2)\on{-wtmod}}

\newcommand{\sWtslA}[1]{L_{#1}(\mf{sl}_2)^A\on{-wtmod}}

\newcommand{\ssWtsl}[1]{\W_{#1}(\ssl)\on{-wtmod}}
\newcommand{\BssWtsl}[1]{\W^{#1}(\ssl)\on{-wtmod}}

\begin{document}

\newcommand{\ssl}{\mathfrak{sl}_{2|1}}
\newcommand{\vir}{\mathcal{M}}
\newcommand{\virp}{\vir_{>0}}
\newcommand{\virb}{\vir_{\ge0}}
\newcommand{\virz}{\vir_0}
\newcommand{\virn}{\vir_{<0}}

\newcommand{\uea}[1]{\mathcal{U}(#1)}

\newcommand{\hw}{highest-weight}
\newcommand{\foh}[1]{\textcolor{blue}{#1}}  %F macro for Flor's comments/changes when reviewing the article based on the referee's suggestions
\newcommand{\jy}[1]{\textcolor{red}{#1}}
\newcommand{\oc}{\mathcal{O}_c}
\newcommand{\ocfin}{\oc^{\textup{fin}}}
\newcommand{\ocleft}{\oc^{\textup{L}}}
\newcommand{\ocright}{\oc^{\textup{R}}}
\newcommand{\cofcat}{\mathcal{C}_1}

%\keywords{Affine vertex algebras; Wakimoto modules, Restricted modules, logarithmic modules}
\title[  ]{Tensor categories of weight modules of $\widehat{\mathfrak{sl}}_2$ at admissible level}
%\subjclass[2000]{ Primary 17B69, Secondary 17B67, 17B68, 81R10}
\author{Thomas Creutzig}
\address{Department of Mathematical and Statistical Sciences, University of Alberta, 632 CAB, Edmonton, Alberta, Canada T6G 2G1}
\email{creutzig@ualberta.ca}

\date{}
\maketitle

\begin{abstract}
 The category of weight modules $\sWtsl{k}$ of the simple affine vertex algebra of $\mathfrak{sl}_2$ at an admissible level $k$ is neither finite nor semisimple and modules are usually not lower-bounded and have infinite dimensional conformal weight subspaces. However this vertex algebra enjoys a duality with $\W_\ell(\ssl)$, the simple prinicipal $\W$-algebra of $\mathfrak{sl}_{2|1}$ at  level $\ell$ (the $N=2$ super conformal algebra) where the levels are related via $(k+2)(\ell+1)=1$.  Every weight module of $\W_\ell(\ssl)$ is lower-bounded and has finite-dimensional conformal weight spaces. 
  The main technical result is that every weight module of $\W_\ell(\ssl)$ is $C_1$-cofinite. The existence of a vertex tensor category follows and the theory of vertex superalgebra extensions implies the existence of vertex tensor category structure on $\sWtsl{k}$ for any admissible level $k$.
    
 As applications, the fusion rules of ordinary modules with any weight module are computed and it is shown that $\sWtsl{k}$ is a ribbon category if and only if $\sWt{k+1}$ is, in particular it follows that for admissible levels $k = - 2 + \frac{u}{v}$ and $v \in \{2, 3\}$ and $u = -1 \mod v$  the category $\sWtsl{k}$ is a ribbon category.
\end{abstract}

\baselineskip=14pt
\newenvironment{demo}[1]{\vskip-\lastskip\medskip\noindent{\em#1.}\enspace
}{\qed\par\medskip}

\section{Introduction}

Let $\mathfrak g$ be a simple Lie algebra with Cartan subalgebra $\mathfrak h$ and let $\mathcal C(\mathfrak g)$ be its category of weight modules, i.e. modules that have semisimple action of $\mathfrak h$ with finite-dimensional weight spaces. $\mathcal C(\mathfrak g)$ contains the category $\mathcal O$ and the category of integrable $\mathfrak g$-modules as subcategories. All modules in $\mathcal C(\mathfrak g)$, except for the integrable ones, are infinite-dimensional and in particular the tensor product of two infinite dimensional modules has infinite-dimensional weight spaces: The category $\mathcal C(\mathfrak g)$ is already not a tensor category. 
The category $\mathcal C(\mathfrak g)$ lacks other niceness conditions, e.g. it doesnot have enough projective objects \cite{Maz10}.

It thus seems to be futile to look at a related problem for the affinization $\widehat{\mathfrak g}$ of $\mathfrak{g}$.  We do so nonetheless. Let
\[
\widehat{\mathfrak g} =  \mathfrak g \otimes_{\mathbb C} \mathbb C[t, t^{-1}] \oplus \mathbb C K, \qquad 
\widehat{\mathfrak g}_{\geq 0} =  \mathfrak g \otimes_{\mathbb C} \mathbb C[t] \oplus \mathbb CK,
\]
with $K$ central.
Then any object $M$ in $\mathcal C(\mathfrak g)$ lifts to a weight module of $\widehat{\mathfrak g}$ at level $k \in \mathbb C$ in the usual way: $M$ becomes a $\widehat{\mathfrak g}_{\geq 0}$ module by setting $K = k \text{Id}$ and $(\mathfrak g \otimes_{\mathbb C} t\mathbb C[t]) M =0$.  Then $\widehat M_k := U(\widehat{\mathfrak g} ) \otimes_{U(\widehat{\mathfrak g}_{\geq 0} )} M$ is a $\widehat{\mathfrak g}$-module at level $k$. The module $\widehat{\mathbb C}_k$ can be given the structure of a vertex operator algebra, the universal affine vertex algebra of $\mathfrak g$ at level $k$, $V^k(\mathfrak g)$. $V^k(\mathfrak g)$-modules are precisely smooth $\widehat{\mathfrak{g}}$-modules at level $k$. Any $V^k(\mathfrak g)$-module is graded by conformal weight and 
let $KL_k(\mathfrak g)$ be the category of $V^k(\mathfrak g)$-modules whose generalized conformal weight spaces are integrable $\mathfrak g$-modules. Kazhdan and Lusztig proved in the 1990's \cite{KazLus93,KazLus94a} that for almost all levels $KL_k(\mathfrak g)$ is a braided tensor category and as such equivalent to the category of finite-dimensional weight modules of the quantum group $U_q(\mathfrak g)$ for $q = \text{exp}\left( \frac{\pi i}{\ell (k+h^\vee}\right)$ with $\ell$ the lacing number and $h^\vee$ the dual Coxeter number of $\mathfrak g$.   

In a major effort Huang-Lepowsky and Zhang established a tensor category theory for vertex operator algebras \cite{HLZ1,HLZ2,HLZ3,HLZ4,HLZ5,HLZ6,HLZ7,HLZ8}. Any vertex tensor category is in particular a braided tensor category. Unfortunately various finiteness criteria need to be satisfied in order for the Huang-Lepowsky-Zhang theory to hold. For example it is necessary that conformal weight spaces are finite dimensional, a condition that holds for $\widehat M_k$ if and only if $M$ is integrable as a $\mathfrak g$-module. 

So far there are two results on vertex tensor category structure on categories that have modules that are neither lower-bounded nor have finite dimensional comformal weight spaces. 
Firstly, Allen and Wood proved this for the $\beta\gamma$-VOA \cite{Allen:2020kkt} and secondly it has been established for the $\mathcal B_p$ algebras of \cite{CreCos13} in \cite{CMY3}. Note that the $\beta\gamma$-VOA is the $\mathcal B_2$-algebra and that the Heisenberg coset of the $\mathcal B_p$-algebra is the singlet algebra $\mathcal M(p)$.  All these tensor category results crucially used that every finite length $\mathcal M(p)$-module is $C_1$-cofinite \cite{Creutzig:2016htk,CMY6}.

For us $C_1$-cofiniteness is also key, albeit for a very different vertex operator algebra. 
The idea of this work is to show that all finiteness conditions of the  Huang-Lepowsky-Zhang theory hold for a dual vertex operator superalgebra and then use the duality to inherit vertex tensor category structure for the affine vertex operator algebra. We consider the special case of $\mathfrak g = \sltwo$ and the category of admissible weight modules, $\sWtsl{\ell}$ of \cite{ACK}.  Here $ L_\ell(\sltwo) $ denotes the simple quotient of $ V^\ell(\sltwo)$ and $\ell$ is a Kac-Wakimoto admissible level for $\sltwo$ \cite{KacWak88}.
This means that $\ell = -2 + \frac{u}{v}$ for coprime positive integers $u, v$ and $u\geq 2$. 
The dual vertex superalgebra is the $N=2$ superconformal algebra which we denote by $\W_k(\ssl)$ as it is the quantum Hamiltonian reduction of the affine vertex operator superalgebra of $\mathfrak{sl}_{2|1}$ at level $k$. Let $\cF$ be the vertex operator superalgebra of a pair of free fermions. The Kazama-Suzuki correspondence (or duality) says that \cite{KS}
\[
\W_k(\ssl) \cong \text{Com}(\pi, L_\ell(\sltwo) \otimes \cF).
\]
Here $\pi$ is a certain Heisenberg vertex operator subalgebra of $L_\ell(\sltwo) \otimes \cF$.
This correspondence is strong in the sense that it gives blockwise equivalences of categories \cite{Feigin:1997ha,Sato,CGNS,CLRW}.
Let  $\ssWtsl{k}$ be the category of weight modules of $\W_k(\ssl)$. We show that this category satisfies all finiteness conditions that are necessary to ensure the existence of a vertex tensor supercategory structure and then use the theory of vertex operator superalgebra extensions \cite{CKM} to infer the corresponding result for $\sWtsl{\ell}$.

\begin{theorem} \textup{(Theorem \ref{thm:main})}
Let $\ell$ be an admissible level for $\sltwo$ and $k$ be defined via $(\ell+2)(k+1) =1$. Then $\ssWtsl{k}$  is a vertex tensor supercategory and $\sWtsl{\ell}$ is a vertex tensor category.
\end{theorem}

The main issue is to show that every object in $\ssWtsl{k}$ satisfies a certain finiteness condition, called $C_1$-cofiniteness. Most of this article is devoted to proving that this is indeed true, see Theorem \ref{keythm} and Corollary \ref{cor:C1}.
Having a vertex tensor category has a few fairly direct consequences. 
\begin{enumerate}
\item The fusion rules of any ordinary module with any projective or simple object in $\sWtsl{\ell}$, see Theorem \ref{thm:fusion} for the concrete formulae. The result is as conjectured via Verlinde's formula \cite{CreMod12, CR1}.
\item We introduce a subcategory  $\sWtslA{\ell}$ of $ \sWtsl{\ell}$ and show that this is a tensor subcategory (Corollary \ref{cor:subcategory}) and 
for $r =1, \dots, u-1$ we define functors 
\[
\mathcal F_r : \sWtslA{\ell+1} \rightarrow \sWtsl{\ell} \boxtimes\sWtsl{1}
\]
and show that they are each fully faithful, see Proposition \ref{prop:Fr}.
\item The previous results are used to establish a translation property, namely 
$\sWtsl{\ell+1}$ is a ribbon category  if and only if $\sWtsl{\ell}$ is a ribbon category, see Corollary \ref{cor:translation}.
\item Thanks to \cite{CMY3} (see Corollary \ref{cor:ribbon}.):
$\sWtsl{\ell}$ is a ribbon category for admissible level $\ell = -2 + \frac{u}{v}$ and $v \in \{2, 3\}$ and $u = -1 \mod v$. Moreover all fusion rules follow from \textup{\cite{CMY3}} using Corollary \ref{cor;tensor}. 
\end{enumerate}

Of course it is expected that our vertex tensor category result holds in much more generality. 
\begin{conjecture} Let $\ell$ be an admissible level for the simple Lie algebra $\g$. 
\begin{enumerate}
\item $L_{\ell}(\g)\on{-wtmod}$ is a vertex tensor category.
\item $L_{\ell}(\g)\on{-wtmod}$ is a ribbon category.
\end{enumerate}
\end{conjecture}
The second point is a strengthening of the first one. We proved the first point of this conjecture for $\g = \sltwo$ and the second one for a special series of levels. 
One in fact might wonder if a similar result holds for affine vertex operator superalgebras. In forthcoming works and using the present work, the existence of vertex tensor category for certain categories of $L_\ell(\mathfrak{osp}_{1|2})$ and $L_{\ell}(\mathfrak{sl}_{2|1})$ at specific series of levels $\ell$ will be established. The reason is that both vertex superalgebras are extensions of $L_\ell(\sltwo)$ times either a Virasoro vertex algebra in the case of $L_\ell(\mathfrak{osp}_{1|2})$  or times a vertex algebra closely related to  $L_k(\sltwo)$ in the case of $L_{\ell}(\mathfrak{sl}_{2|1})$. In the latter case, $\ell$ and $k$ are related via $(\ell+1)(k+1)=1$ and it is an example of Feigin-Semikhatov duality \cite{FeiW2n04,CL22}.
The case of $L_\ell(\mathfrak{osp}_{1|2})$ has already been studied in \cite{Creutzig:2018ogj}  under the assumptions of existence of vertex tensor category as well as our fusion rule results of Theorem \ref{thm:fusion}.

It is widely believed that vertex tensor categories can often also be realized as categories of modules of some (quasi) Hopf algebra. In the case of $\sWtsl{\ell}$ a close connection to a quantum group of $\mathfrak{sl}_{2|1}$ is expected thanks to the theory of \cite{CLR} together with the free field realization of \cite{Ad1}.  The pioneer of this idea (or better program) is Alyosha Semikhatov \cite{Semikhatov:2013xk,Semikhatov:2011ie,Semikhatov:2011uc}. We summarize section 6 of \cite{CLR} for the case of $\g = \mathfrak{sl}_{2|1}$.
$\mathfrak{sl}_{2|1}$ allows for a triangular decomposition $\mathfrak{sl}_{2|1} = \g_- \oplus \g_0 \oplus \g_+$ with $\g_0 = \mathfrak{gl}_2 =  \mathfrak{sl}_2 \oplus \mathbb C h$ and $\g_\pm = \mathbb C^2_{\pm 1}$ each odd and each a two-dimensional simple $\g_0$-module of $h$-weight $\pm 1$. Consider $N:= U_q(\mathfrak{gl}_2\oplus \mathbb C^2_1)$ as a Nichols algebra in the category $\mathcal U_q(\mathfrak{gl}_2)$ of tilting modules of $U_q(\mathfrak{gl}_2)$. By this we mean the Deligne product of the category of tilting modules of $U_q(\sltwo)$ and the category of graded vector spaces $\text{Vect}^Q_{\mathbb C}$ with some specific quadratic form $Q$ depending on $q$.  
Then the category of tilting modules of $U_q(\mathfrak{sl}_{2|1})$ is the category of Yetter-Drinfeld modules for the Nichols algebra $N$ inside $\mathcal U_q(\mathfrak{gl}_2)$. This construction can be slightly modified and the following procedure one might like to call a partial semisimplification. Firstly the Nichols algebra is $N \cong  \mathbb C_0 \oplus \mathbb C^2_1 \oplus \mathbb C_{2}$  where $\mathbb C^n_\lambda$ is the $n$-dimensional irreducible module of $U_q(\sltwo)$ of $\mathbb C$-weight $\lambda$. We can rescale the weights by another complex number $\mu$, that is set $N_\mu :=  \mathbb C_0 \oplus \mathbb C^2_\mu \oplus \mathbb C_{2\mu}$.  
 Secondly if $q$ is a root of unity then we can replace the category of tilting modules of $U_q(\sltwo)$ by its semisimplification $\widetilde U_q(\sltwo)$, that is we quotient by all negligible morphisms, see e.g. Theorem 5 of \cite{Sawin}. Let $\widetilde N_\mu$ be the image of $N_\mu$ under semisimplification. 
  Then one can define $\widetilde{\mathcal  U}^\mu_q(\mathfrak{sl}_{2|1})$ to be the category of Yetter-Drinfeld modules for $\widetilde N_\mu$ inside $\widetilde{\mathcal U}_q(\sltwo) \boxtimes \text{Vect}^Q_{\mathbb C}$.
\begin{conjecture}\label{conj:KL}
Let $\ell$ be admissible and let $q = \text{exp}\left( \pi i (\ell+2) \right)$. Then there exists $\mu \in \mathbb C$ such that
\[
\sWtslA{\ell}  \cong \widetilde{\mathcal U}^\mu_q(\mathfrak{sl}_{2|1})
\]
as braided tensor categories. 
\end{conjecture} 
\noindent We only consider the subcategory $\sWtslA{\ell}$ of $\sWtsl{\ell}$. $\sWtsl{\ell}$ should then be closely related to $ \widetilde{\mathcal U}^\mu_q(\mathfrak{sl}_{2|1})$ times a semisimplification of quantum $\sltwo$. The proof of the Conjecture amounts to applying a slight generalization of Theorem 8.1 of \cite{CLR}, that is to effectively prove a few abelian equivalences. Here is a list of the main  open problems concerning $\sWtsl{\ell}$ at admissible level $\ell$.
\begin{enumerate}
\item Prove rigidity for all $\ell$.
\item Compute all fusion rules.
\item Prove Conjecture \ref{conj:KL}.
\end{enumerate}
It should be possible to verify all assumptions of Theorem 8.1 of \cite{CLR} once the first two points of this list are proven, i.e. the last point should follow from the first two together with a generalization Theorem 8.1 of \cite{CLR}. 
We believe that all these three points are now within reach, e.g. Florencia Orosz~Hunziker informed me that the fusion rules will be largely settled soon \cite{flor}.

\subsection{Organization}

In section \ref{VOA} we introduce few basics on vertex operator modules, in particular we prove a non-degeneracy property of intertwining operators that is needed later, see Proposition \ref{prop:intertwiner}. In sections \ref{sec:vir}, \ref{sltwo} and \ref{N=2} some Virasoro algebra, affine $\sltwo$ and $N=2$ super Virasoro algebra results are recalled. Also a few new statements that are needed later are proved. The main new insight is Theorem \ref{thm:fusion} about the fusion rules of ordinary modules with any simple or projective module. 
Next some relevant free field realizations are introduced in section \ref{ff} and then necessary results on $C_1$-cofiniteness in section \ref{C1}. With this all ingredients are in place and the $C_1$-cofiniteness of all objects in $\ssWtsl{k}$ can be proven see Theorem \ref{keythm} and Corollary \ref{cor:C1}.
In section \ref{tc} we show the existence of vertex tensor categories and collect interesting consequences.

\subsection*{Acknowledgements} I appreciate useful discussions with F. Orosz~Hunziker, R. McRae, T. Arakawa, K. Kawasetsu, J. Yang, and especially D. Adamovic and D. Ridout. A big thank you to J. Yang for having given me the opportunity  to give a lecture series on this topic.

\section{Vertex operator algebra modules}\label{VOA}

Some definitions are recalled, for a detailed overview see \cite{H2}.

Let $(V, Y, {\bf 1}, \omega)$ be a vertex operator algebra. We first recall the definitions of $V$-modules.
\begin{definition}
  A \textit{weak $V$-module} (or simply {\em $V$-module}) is a vector space $W$
equipped with a vertex operator map
$$
 Y_W:  V  \rightarrow  (\mathrm{End}\,W)[[x,x^{-1}]] , \qquad
    v  \mapsto  Y_W(v,x)=\sum\limits_{n\in\mathbb{Z}} v_n\,x^{-n-1}\\
$$
satisfying the following axioms:
\begin{itemize}
\item[(i)] The lower truncation condition: For $u\in V$, $w \in W$, $Y_W(u,x)w$ has only finitely many
terms of negative powers in $x$.

 \item[(ii)] The vacuum property:
\begin{equation*}
 Y_W(\mathbf{1},x)=1_W.
 \end{equation*}

 \item[(iii)] The Jacobi identity: for $u,v\in V$,
\begin{align*}
 x_0^{-1}&\delta\left(\dfrac{x_1-x_2}{x_0}\right) Y_W(u,x_1)Y_W(v,x_2) 
 - x_0^{-1}\delta\left(\dfrac{-x_2+x_1}{x_0}\right)Y_W(v,x_2)Y_W(u,x_1)
\\
 & = x_2^{-1}\delta\left(\dfrac{x_1-x_0}{x_2}\right)Y_W(Y(u,x_0)u,x_2).
\end{align*}

\item[(iv)] The Virasoro algebra relations: if we write $Y_W(\omega,x)=\sum_{n\in\mathbb{Z}} L_n x^{-n-2}$,
then for any $m,n\in\mathbb{Z}$,
\begin{equation*}
 [L_m,L_n]=(m-n)L_{m+n}+\dfrac{m^3-m}{12} \delta_{m+n,0} c,
\end{equation*}
where $c$ is the central charge of $V$.

\item[(v)] The $L_{-1}$-derivative property: for any $v\in V$,
\begin{equation*}
 Y_W(L_{-1}v,x)=\dfrac{d}{dx} Y_W(v,x).
 \end{equation*}
\end{itemize}
\end{definition}

Let $V$ be a vertex operator algebra and let $(W, Y_W)$ be a  $V$-module with
\[
W = \coprod_{n \in \mathbb{C}}W_{[n]}
\]
where for $n \in \mathbb{C}$, $W_{[n]}$ is the generalized weight space with weight $n$. Its {\em contragredient module} is the vector space
\[
W' = \coprod_{n \in \mathbb{C}}(W_{[n]})^*,
\]
equipped with the vertex operator map $Y'$ defined by
\[
\langle Y'(v,x)w', w \rangle = \langle w', Y^{\circ}_W(v,x)w \rangle
\]
for any $v \in V$, $w' \in W'$ and $w \in W$, where
\[
Y^{\circ}_W(v,x) = Y_W(e^{xL(1)}(-x^{-2})^{L_0}v, x^{-1}),
\]
for any $v \in V$, is the {\em opposite vertex operator} (cf. \cite{FHL}).
%We also use the standard notation
%\[
%\overline{W} = \prod_{n \in \mathbb{C}}W_{[n]},
%\]
%for the formal completion of $W$ with respect to the $\mathbb{C}$-grading.

Sometimes in the defintion of modules the Jacobi identity is replaced by duality, that is in terms of commutativity and associativity of correlation functions. 
\subsection{Intertwining Operators}

We also need the notion of logarithmic intertwining operators:
\begin{definition}\label{log:def}
{\rm
Let $(W_1,Y_1)$, $(W_2,Y_2)$ and $(W_3,Y_3)$ be $V$-modules. A {\em logarithmic intertwining
operator of type ${W_3\choose W_1\,W_2}$} is a linear map
\begin{align}\label{log:map0} \nonumber
\mathcal{Y}: &W_1\otimes W_2\to W_3\{x\}[\log x] \\ \nonumber
&w_{(1)}\otimes w_{(2)}\mapsto{\mathcal{Y}}(w_{(1)},x)w_{(2)}=\sum_{k=0}^{K}\sum_{n\in
{\mathbb C}}{w_{(1)}}_{n, k}^{
\mathcal{Y}}w_{(2)}x^{-n-1}(\log x)^{k}
\end{align}
for all $w_{(1)}\in W_1$ and $w_{(2)}\in W_2$, such that the
following conditions are satisfied:
\begin{itemize}
\item[(i)] The {\em lower truncation
condition}: for any $w_{(1)}\in W_1$, $w_{(2)}\in W_2$, $n\in
{\mathbb C}$ and $k=0, \dots, K$,
\begin{equation} \nonumber
{w_{(1)}}_{n+m, k}^{\mathcal{Y}}w_{(2)}=0\;\;\mbox{ for }\;m\in {\mathbb
N} \;\mbox{sufficiently large}.
\end{equation}
\item[(ii)] The {\em Jacobi identity}:
\begin{equation} \nonumber
\begin{split}
 x^{-1}_0\delta \bigg( {x_1-x_2\over x_0}\bigg)
Y_3(v,x_1){\mathcal{Y}(w_{(1)},x_2)w_{(2)}}
- x^{-1}_0\delta \bigg( {x_2-x_1\over -x_0}\bigg)
{\mathcal{Y}}(w_{(1)},x_2)Y_2(v,x_1)w_{(2)} \\
 = x^{-1}_2\delta \bigg( {x_1-x_0\over x_2}\bigg){
\mathcal{Y}}(Y_1(v,x_0)w_{(1)},x_2) w_{(2)}
\end{split}
\end{equation}
for $v\in V$, $w_{(1)}\in W_1$ and $w_{(2)}\in W_2$.
\item[(iii)] The {\em $L_{-1}$-derivative property:} for any
$w_{(1)}\in W_1$,
\begin{equation}\nonumber
{\mathcal{Y}}(L_{-1}w_{(1)},x)=\frac d{dx}{\mathcal{Y}}(w_{(1)},x).
\end{equation}
\end{itemize}
A logarithmic intertwining operator is called an {\it intertwining operator} if no $\log x$ appears in ${\mathcal{Y}}(w_{(1)},x)w_{(2)}$
for $w_{(1)}\in W_1$ and $w_{(2)}\in W_2$. The dimension of the space  $\mathcal{V}_{W_1 W_2}^{W_3}$ of all logarithmic intertwining
operators is called the {\it fusion rule}.}
\end{definition}
As for modules, the Jacobi identity can be replaced by duality, that is \cite{H2}
for any $w_1 \in W_1, w_2 \in W_2, w_3' \in W_3'$ and $v\in V$ and for any single valued branch $\ell(z_2)$ of the logarithm of $z_2$ in the region $z_2 \neq 0, 0 \leq \text{arg} z_2 \leq 2\pi$, the series
\begin{equation}\nonumber
\begin{split}
\left< w_3', Y_3(v, z_1) \mathcal{Y}(w_1, x_2) w_2 \right> \Big\vert_{x_2^n = e^{n \ell(z_2)}, n \in \mathbb C} \\ 
\left< w_3',  \mathcal{Y}(w_1, x_2)Y_2(v, z_1) w_2 \right> \Big\vert_{x_2^n = e^{n \ell(z_2)}, n \in \mathbb C} \\ 
\left< w_3',  \mathcal{Y}\left(Y_1(v, z_1 -z_2)w_1, x_2 \right) w_2 \right> \Big\vert_{x_2^n = e^{n \ell(z_2)}, n \in \mathbb C} \\ 
\end{split}
\end{equation}
are absolutely convergent in the regions $|z_1| > |z_2| > 0, |z_2| > |z_1| > 0, |z_2| > |z_1 - z_2| > 0$, respectively, to a common analytic function in $z_1$ and $z_2$ and can be analytically extended to a multivalued analytic function with the only possible poles at $z_1 =0$ and $z_1 =z_2$ and the only possible branch point $z_2 =0$. 

For any $r \in \mathbb Z$ and any intertwining operator  $\mathcal Y$ of type ${W_3\choose W_1\,W_2}$, one defines skew symmetry via
\begin{equation}\label{skew}
\Omega_r(\mathcal Y)(w_2, x)w_1 = e^{xL_{-1}} \mathcal Y(w_1, y)w_2 \Big\vert_{y^n = e^{n(2r+1) \pi i} x^n, n \in \mathbb C, \ln y = \ln x + (2r+1)\pi i}
\end{equation}
and has that $\Omega_r(\mathcal Y)$ is an interwining operator of type ${W_3\choose W_2\,W_1}$. We will need the following statement. 
\begin{proposition}\label{prop:intertwiner}
Let $\mathcal Y$ be an intertwining operator of type ${W_3\choose W_1\,W_2}$ and let $w_1 \in W_1, w_2 \in W_2$ and denote by $\left< w_1\right>$ and $\left< w_2\right>$
the $V$-modules generated by $w_1$ and $w_2$. If $\mathcal Y(w_1, z)w_2 =0$, then $\mathcal Y(v_1, z)v_2 =0$ for all $v_1 \in \left< w_1\right>$ and $v_2 \in \left< w_2\right>$.
\end{proposition}
\begin{proof}
If $\mathcal Y(w_1, z)w_2 =0$, then 
\[
\left< w_3', Y_3(v, z_1) \mathcal{Y}(w_1, x_2) w_2 \right>  = 0
\]
for all $w_3'\ \in W_3'$ and all $v\in V$ and so 
\[
\left< w_3',  \mathcal{Y}(w_1, x_2)Y_2(v, z_1) w_2 \right> = 0 
\]
for all $w_3'\ \in W_3'$ and all $v\in V$  as well. Hence 
\[
 \mathcal{Y}(w_1, x_2)Y_2(v, z_1) w_2 = 0 
\]
for all $v\in V$ in particular 
\[
 \mathcal{Y}(w_1, x_2) v_n w_2 = 0 
\]
for all $v\in V$ and $n \in \mathbb Z$, i.e. 
\[
 \mathcal{Y}(w_1, x_2) v_2 = 0
\]
for all $v_2 \in \left< w_2\right>$.
By skew-symmetry \eqref{skew} it follows that
\[
\Omega_r(\mathcal Y)(v_2, x)w_1  = 0
\]
for all $v_2 \in \left< w_2\right>$. Now applying the exact same argument as above to  $\Omega_r(\mathcal Y)$ gives 
\[
\Omega_r(\mathcal Y)(v_2, x)v_1  = 0
\]
for all $v_1 \in \left< w_1\right>$ and $v_2 \in \left< w_2\right>$ and applying skew-symmetry again yields $\mathcal Y(v_1, z)v_2 =0$ for all $v_1 \in \left< w_1\right>$ and $v_2 \in \left< w_2\right>$.
\end{proof}

\section{Virasoro vertex algebra} \label{sec:vir}

We review the Virasoro vertex algebra using \cite{C-Vir}.
The Virasoro algebra is the Lie algebra with basis $L_n$ for $n \in \ZZ$ and $C$
 with the commutation relations
\begin{equation}
[L_m, L_n] = (m-n)L_{m+n}+\frac{m^3-m}{12}\delta_{m+n,0} C,
\end{equation}
and $C$ is central. Let $h, c \in \CC$, then the Verma module $V(c, h)$ of highest weight $h$ and central charge $c$ is 
the module generated by a vector $\one_{c, h}$, such that 
\[
L_n \one_{c, h} = 0 \quad \text{for} \  n > 0, \qquad  L_0 \one_{c, h} = h \one_{c. h} \quad \text{and} \quad C \one_{c, h} = c \one_{c, h}
\]
and such that the $L_n$ for negative $n$ act freely on it. 
$V(c, h)$ need not be simple, for example the vector $L_{-1}\one_{c, 0}$ is a singular vector that generates a proper submodule, call it $S(c, 0)$. This quotient will be denoted by
\[
\vir^k = V(c, 0)/S(c, 0). 
\]
Where we take $c$ to be of the form $c=c(t)=13-6t-6t^{-1}$ and set $k = t-2$. 
Frenkel and Zhu showed that $\vir^k$ can be given the structure of a vertex algebra, the universal Virasoro vertex algebra of central charge $c$ \cite{FZ1}. 
We will denote this vertex algebra also by $\vir^k$ and by $\vir_k$ its simple quotient. 

%For the Verma module $V(c, h)$ we denote its simple quotient by $L(c, h)$. 
Singular vectors are understood in general: 
\begin{proposition}\textup{\cite{FF}}\label{thm:vermared}
For
$r,s\in \ZZ_{\ge1}$ and $t \in \CC\setminus\{0\}$, define
\begin{equation} \label{eq:virpara}
c=c(t)=13-6t-6t^{-1}, \quad
h=h_{r,s}(t)=\frac{s^2-1}{4}t-\frac{rs-1}{2}+\frac{r^2-1}{4}t^{-1}.
\end{equation}
\begin{enumerate}
	\item If there exist $r,s\in \ZZ_{\ge1}$ and $t \in \CC\setminus\{0\}$ such that $c$ and $h$ satisfy \eqref{eq:virpara}, then there is a singular vector of weight $h+rs$ in the Verma module $V(c,h)$.
	\item Conversely, if $V(c,h)$ possesses a non-trivial singular vector, then there exist $r, s \in \ZZ_{\ge1}$ and $t \in \CC \backslash \{0\}$ such that \eqref{eq:virpara} holds.
	\item For each $N$, there exists at most one singular vector of weight $h+N$ in $V(c,h)$, up to scalar multiples.
\end{enumerate}
\end{proposition}
The set of special weights is denoted by $H_c$
\[
H_c = \left\{ h_{r, s}(t) =\frac{s^2-1}{4}t-\frac{rs-1}{2}+\frac{r^2-1}{4}t^{-1}\, \Big| \, r,s\in \ZZ_{\ge1} \right\}.
\]

\begin{proposition}\textup{\cite{As}}\label{sing}
If $V(c,h)$ has a
singular vector $v$ of conformal weight $h+N$,
then
\begin{align}
v=\sum_{ |I| =N}a_{I}(c,h)L_{-I}\one_{c,h},
\end{align}
where $L_{-I}=L_{-i_{1}}\cdots L_{-i_{n}}$ and the sum is over sequences $I=\{i_1, \dots, i_n\}$ of ordered $n$-tuples $i_1\ge \dots\ge i_n$ with $ |I| =i_1+\dots+i_n=N$. Moreover, the
coefficients $a_{I}(c,h)$ depend polynomially on $c$ and $h$ and the coefficient $a_{\{1,\dots,1\}}(c,h)$ of $L_{-1}^{N}$ may be chosen to be $1$.
\end{proposition}

In general we will denote the simple quotient of $V(c,h)$ by $L(c,h)$.
For $r,s\in \ZZ_{\ge1}$ and $t \in \CC\setminus\{0\}$ we denote the simple quotient of $V(c(t), h_{r, s}(t))$ also
by  $\vir^{t-2}_{ r, s}$, that is $L(c(t), h_{r, s}(t))= \vir^{t-2}_{ r, s}$, and if the central charge is clear, then we just write $\vir_{r, s}$.  

Let $t = \frac{u}{v}$ with $u, v  \in \ZZ_{\geq 2}$ and coprime and set
$k = t-2$.
Then  $\vir_k$ is rational \cite{W} and a complete list of simple modules is given by $\vir^k_{r, s}$ with $1\leq r \leq u-1$, $1\leq s \leq v-1$. In addition there are the isomorphisms $\vir^k_{r, s} \cong \vir^k_{u-r, v-s}$ and the fusion rules are given by
\begin{equation}\label{eq:fus-vir}
\begin{split} 
\vir^k_{r, s} \boxtimes \vir^k_{r', s'} &\cong \bigoplus_{r'' =1}^{u-1} \bigoplus_{s'' =1}^{v-1} N^{u, v\quad (r'', s'')}_{(r, s), (r', s')} \vir^k_{r'', s''} \\
N^{u, v'\quad (r'', s'')}_{(r, s), (r', s')} &= N^{u\ \  r''}_{r, r'}N^{v\ \  s''}_{s, s'} 
\end{split}
\end{equation}
\[
N^{q\ \ t''}_{t, t'} = \begin{cases} 1 & \text{if} \ |t-t'| + 1 \leq t'' \leq \text{min}\{ t+t'-1, 2q - t -t' -1\} \ \text{and} \ t+t'+t'' \ \text{odd} \\ 0 & \text{else} \end{cases}
\]

 \section{The affine vertex operator algebra of $\sltwo$ }\label{sltwo}

We fix $k = - 2 + \frac{u}{v}$ with $u, v$ co-prime positive integers, both at least equal to two, so that $k$ is admissible for $\sltwo$. 
Let $\{ f, h, e\}$ be a basis of $\sltwo$ with commutation relations
\[
[h, e] = 2e, \qquad [h, f] -2f, \qquad [e, f] = h.
\]
The Casimir in this basis is \[
C = \frac{hh}{2} + ef  + fe.
\]
The affinization $\asltwo$ has basis $\{ f_n, h_n, e_n |  n \in\mathbb Z\} \cup \{ K, d \}$ with $K$ central, $d$ the usual derivation and the additional non-zero commutation relations are
\begin{equation}
\begin{split}
[h_m, e_n] &= 2 e_{m+n}, \ \qquad [h_m, h_n] = 2m \delta_{m+n, 0}K,  \\
[h_m, f_n] &= -2 f_{m+n}, \qquad [e_m, f_n] = h_{m+n} + m\delta_{m+n, 0}K. 
\end{split}
\end{equation}
Affine Weyl translation induce spectral flow automorphisms $\{ \sigma^\ell | \ell \in \mathbb Z\}$  which act explicitely as 
\begin{equation}
\sigma^\ell(e_n) = e_{n - \ell}, \qquad \sigma^\ell(h_n) = h_n + \delta_{n, 0} \ell K, \qquad \sigma^\ell(f_n) = f_{n+\ell}.
\end{equation}
The affine vertex operator algebra $L_k(\sltwo)$ is then strongly generated by the fields
\[
f(z) = \sum_{n \in \mathbb Z} f_n z^{-n-1}, \qquad 
h(z) = \sum_{n \in \mathbb Z} h_n z^{-n-1}, \qquad 
e(z) = \sum_{n \in \mathbb Z} e_n z^{-n-1}, \qquad 
\]
and $K$ acts by the scalar $k \in \mathbb C$ on the underlying vector space.

\subsection{Modules}

 Let $\omega$ be the fundamental weight of $\sltwo$.
 We are interested in three types of $\sltwo$-modules. Firstly the integrable simple modules $L_r$ of highest-weight $(r-1)\omega$ and dimension $r$; secondly for $\lambda \in \mathbb C$, the 
  irreducible highest-weight (lowest-weight) representation of  highest-weight (lowest-weight) $\lambda\omega$ is denoted by $D^+_{\lambda}$ ($D^-_\lambda$);  and thirdly 
 dense  weight modules of weight $\lambda\omega$ and Casimir eigenvalue $\Delta$ will be denoted by $R_{\lambda, \Delta}$ for generic $(\lambda, \Delta)$. These modules have basis $\{v_{\lambda +2n} |  n \in \mathbb Z\}$ and the $\sltwo$ action is given by
  \[
  C v_{\lambda+2n} = \Delta v_{\lambda +2n}, \qquad
  h  v_{\lambda+2n} = (\lambda +2n) v_{\lambda +2n}, \qquad
  e v_{\lambda+2n} = v_{\lambda +2(n+1)} \qquad
  \]
 and these completely specify the action of $f$:
 \begin{equation}\nonumber
 \begin{split}
 f v_{\lambda+ 2(n+1)} &= ef v_{\lambda+2n} = \frac{1}{2}\left(C - \frac{hh}{2} - h\right) v_{\lambda+2n} \\ &= \frac{1}{4} \left( 2 \Delta - (\lambda+2n)(\lambda+2(n+1)) \right)v_{\lambda+2n}. 
 \end{split}
 \end{equation}
 
  These types of modules belong to a coherent family parameterized by $\h/Q^\vee \cong \mathbb C/2\mathbb Z$. Generically these modules are simple and at special values there are  two modules $R^\pm_{\lambda, \Delta}$, here the superscript indicates that the dense module has a highest ($+$) or lowest ($-$) weight module as submodule. For example in the instance of $\lambda$ not dominant integral and provided that $\Delta = \Delta_\lambda$ is the Casimir eigenvalue on $D^+_\lambda$, then $R^+_{\lambda, \Delta_\lambda}$ is the non-trivial extension
\[
0 \rightarrow D^+_\lambda \rightarrow R^+_{\lambda, \Delta_\lambda} \rightarrow D^-_{\lambda + 2} \rightarrow 0, \qquad \in \ \Ext^1_\cD(D^-_{\lambda + 2}, D^+_\lambda).
\]
Let $M$ be an $\sltwo$-module, then we denote by $\widehat{M}^k$ the generalized Verma module of $\asltwo$ at level $k$ and by $\widehat{M}_k$
its almost simple quotient, that is the quotient by all submodules that  intersect the top level of $\widehat{M}^k$ trivially.
The common notation for modules of the affine \voa{}  is
\[
\widehat{L}_{r,k} = \slirr{r}^k, \qquad \widehat{D}^{\pm}_{\lambda_{r,s}, k} = \sldis{r, s}^{\pm, k}, \qquad
\widehat{R}_{\lambda, \Delta_{r, s}, k} = \slrel{\lambda}{r, s}^k
 \qquad
\widehat{R}^\pm_{\lambda, \Delta_{r, s}, k} = \slrel{\lambda}{r, s}^{\pm, k}
\]
with the short-hand notations
\begin{equation}
	k+2 = t = \frac{u}{v}, \qquad u \in \ZZ_{ \geq  2}, \quad v \in \ZZ_{> 1}, \quad \text{gcd} (u,v) = 1.
\end{equation}
and introduce
\begin{equation} \label{eq:DefLambda}
		\lambda_{r,s} = r - 1 - ts.
	\end{equation}
	and 
	\begin{equation} \label{eq:DefDelta}
	\Delta_{r,s} = \frac{(r-ts)^2-1}{4t} = \frac{(vr-us)^2-v^2}{4uv}
\end{equation}
If the level $k$ is clear, then we will omit the superscript. 
Finally for a module $\widehat{M}$ we denote by $\sigma^{\ell}(\widehat{M})$ the module on which the action is twisted by $\sigma^{\ell}$, that is the underlying vector space is isomorphic to $\widehat{M}$ with isomorphism denoted by $\sigma^\ell$ as well and action for $x \in \asltwo$ and $m \in \widehat{M}$
\[
x. \sigma^\ell(m) = \sigma^{\ell}\left( \sigma^{-\ell}(x).m\right).
\]
Let $\sWtsl{k}$ be the category of weight modules of the simple affine vertex operator algebra of $\sltwo$ as defined in \cite{ACK}. That is 
 the category of finitely generated smooth weight modules at level $k$ with finite-dimensional weight spaces  that are $L_k(\sltwo)$-modules. This means  $L_k(\sltwo)$-modules $\mathcal M$ that are finitely generated smooth $\widehat{\sltwo}$-modules 
 at level $k$ such that the Cartan subalgebra $\mathfrak{h}=\mathbb C h_0$ acts semisimply and so $\mathcal M$ is graded by conformal weight and $\mathfrak{h}$, that is 
\[
\mathcal M = \bigoplus_{\lambda, \Delta} \mathcal M_{\lambda, \Delta}
\]
and for each $\lambda$ there exists $h_\lambda$, such that $\mc{M}_{\lambda, \Delta} = 0$ for $\text{Re}(\Delta) < \text{Re}(h_\lambda)$
and
$\text{dim}   \mathcal M_{\lambda, \Delta} < \infty$ for any $(\lambda, \Delta)$.

The classification of simple modules is essential due to Adamovi\'c and Milas; though  that every simple module is the spectral flow image of a lower bounded is easily shown in \cite{ACK}.
\begin{theorem}[Adamovi\'c-Milas \cite{AdM-MRL}, see also \cite{RW2}] \label{rhwsimples}
	Let $k = -2 + \frac{u}{v}$ be an admissible level. Then, 
	the simple objects in $\sWtsl{k}$ %$\sWtpgs{k}$
	%the irreducible relaxed \hw{} $\slvoa{k}$-modules 
	are exhausted, up to isomorphism, by the following list:
	\begin{itemize}
		\item The $\sigma^\ell(\slirr{r})$, for $r = 1,\dots,u-1$ and $\ell \in \mathbb Z$;
		\item The $\sigma^\ell(\sldis{r,s}^\pm)$, for $r = 1,\dots,u-1$ and $s = 1,\dots,v-1$ and $\ell \in \mathbb Z$;
		\item The $\sigma^\ell(\slrel{\lambda}{r,s})$, for $r = 1,\dots,u-1$, $s = 1,\dots,v-1$ and $\lambda \in \alg{h}^*$ with $\lambda \neq \lambda_{r,s}, \lambda_{u-r,v-s} \pmod{\rlat}$  and $\ell \in \mathbb Z$.
		\end{itemize}
The only  identifications are $\sigma^\ell(\slrel{\lambda}{r,s}) = \sigma^\ell(\slrel{\lambda}{u-r,v-s})$, $\sigma^\ell(\slrel{\lambda}{r,s}) = \sigma^\ell(\slrel{\mu}{r,s})$, if $\lambda = \mu \pmod{\rlat}$,  $\sigma^\ell(\slirr{r})  = \sigma^{\ell \mp 1}  (\sldis{u-r,v-1}^\pm)$ and $\sigma^\ell(\sldis{r,s}^+) = \sigma^{\ell+1}(\sldis{u-r,v-s-1}^-)$ and here $s\neq v-1$. 
\end{theorem}
The simple $\sigma^\ell(\slrel{\lambda}{r,s})$ are also projective \cite{ACK}, while the projective cover $\sigma^\ell(\slproj{r,s}^\pm)$ of $\sigma^\ell(\sldis{r,s}^\pm)$ has four composition factors and we will describe it in a moment. 
Firstly we have
\begin{proposition} \label{existelle}
\textup{\cite{Ad1, Kawrel19}}
There exist indecomposable modules $\sfmod{\ell}{\slindrel{r,s}^+}$ fitting in the non-split exact sequences
	\begin{equation} \label{es:DED}
	\begin{split}
		\dses{\sfmod{\ell}{\sldis{r,s}^+}}{\sfmod{\ell}{\slindrel{r,s}^+}}{\sfmod{\ell}{\sldis{u-r,v-s}^-}}, \\
		\dses{\sfmod{\ell}{\sldis{r,s}^-}}{\sfmod{\ell}{\slindrel{r,s}^-}}{\sfmod{\ell}{\sldis{u-r,v-s}^+}},
		\end{split}
	\end{equation}
for $\ell \in\ZZ, r = 1, \dots,  u-1, s = 1, \dots, v-1$. 	
\end{proposition}

The complete classification of indecomposable modules is achieved in \cite{ACK}, which we will now recall. There exist indecomposable projective and injective modules $\sfmod{\ell}{\slproj{r, s}}$ satisfying
\begin{equation}\label{ses_proj}
\begin{split}
&\dses {\sfmod{\ell}{\slindrel{r,s}^+}} {\sfmod{\ell}{\slproj{r, s}}}    {\sfmod{\ell+1}{\slindrel{r,s+1}^+}},\\
&\dses {\sfmod{\ell}{\slindrel{r,v-1}^+}} {\sfmod{\ell}{\slproj{r, v-1}}}    {\sfmod{\ell+2}{\slindrel{u-r,1}^+}},\\
&\dses {\sfmod{\ell}{\slindrel{r,s}^-}} {\sfmod{\ell}{\slproj{r, s}}}    {\sfmod{\ell-1}{\slindrel{r,s+1}^-}} ,\\
&\dses {\sfmod{\ell}{\slindrel{r,v-1}^-}} {\sfmod{\ell}{\slproj{r, v-1}}}    {\sfmod{\ell-2}{\slindrel{u-r,1}^-}} ,
\end{split}
\end{equation}
for $\ell \in\ZZ, r = 1, \dots,  u-1, s = 1, \dots, v-2$.

Each simple $\sigma^\ell(\slrel{\lambda}{r,s})$  belongs to a block $E^\ell_{r, s, \lambda}$ whose objects are just direct sums of several copies of the $\sigma^\ell(\slrel{\lambda}{r,s})$.
Set
$$
L^{r,n}_{-\ell(v-1)-m}:=\sigma^{-n-2\ell v-m}(\mc{D}^+_{\phi_{r,\ell},v-1-m}),\qquad  P^{r,n}_{-\ell(v-1)-m}:=\sigma^{-n-2\ell v-m}(\mc{P}_{\phi_{r,\ell},v-1-m})
$$ 
for $
\ell\in\ZZ, 0\leq m\leq v-2,$
with $\phi_{r,\ell}= r $ if $\ell$ is even and $\phi_{r,\ell}= u-r $ if $\ell$ is odd. 
Then there is a block $C_{r, s}$
of  $\sWtsl{k}$ %$\wts k\g$ 
whose simple objects are the 
$L^{r,n}_m$ $(m\in\ZZ)$ and all non-zero $\Ext^1$
among simples are of the form
$$
\Ext^1(L^{r,n}_m,L^{r,n}_{m\pm1})=\C \quad (m\in\ZZ).
$$
The projective  cover and injective hull of $L_{m}^{r, n}$ is denoted by $P_m^{r, n}$ and its Loewy diagram is
\begin{center}
\begin{tikzpicture}[scale=1]
\node (top) at (0,2) [] {$L_m^{r, n}$};
\node (left) at (-2,0) [] {$L_{m-1}^{r, n}$};
\node (right) at (2,0) [] {$L_{m+1}^{r, n}$};
\node (bottom) at (0,-2) [] {$L_m^{r, n}$};
\draw[->, thick] (top) -- (left);
\draw[->, thick] (top) -- (right);
\draw[->, thick] (left) -- (bottom);
\draw[->, thick] (right) -- (bottom);
\node (label) at (0,0) [circle, inner sep=2pt, color=white, fill=black!50!] {$P_m^{r, n}$};
\end{tikzpicture}
\end{center}

\begin{theorem}\label{classification} \textup{\cite{ACK}}
The complete list of the blocks of  $\sWtsl{k}$ %$\wts k\g$ 
are
\begin{align*}
&C_{r,n}&&(1\leq r\leq u-1,0\leq n\leq v-1),\\
&
E^n_{r,s,\lambda}&&(n\in\ZZ ,1\leq r\leq u-1,1\leq s\leq v-1,\lambda\in\C/2\ZZ
\mbox{ with }
\lambda\neq \lambda_{r,s},\lambda_{u-r,v-s}\mbox{ mod 2}).
\end{align*}
\end{theorem}

We will later study the following subcategory.
\begin{definition}
Let $\sWtslA{k}$  be the full subcategory of $\sWtsl{k}$  whose blocks are 
\begin{align*}
&C_{1,n}&&( 0\leq n\leq v-1),\\
&
E^n_{1,s,\lambda}&&(n\in\ZZ,1\leq s\leq v-1,\lambda\in\C/2\ZZ
\mbox{ with }
\lambda\neq \lambda_{1,s},\lambda_{u-1,v-s}\mbox{ mod 2}).
\end{align*}
\end{definition}

\subsection{Branching Rules}

Let $k = -2 + \frac{u}{v}$ be an admissible level as before. 
Recall, that the Virasoro algebra has the coset realization as  a diagonal coset \cite{GKO, ACL}, in particular
\[
\slirr{1}^k \otimes \slirr{1}^1  \cong \bigoplus_{\substack{ m=1 \\ m \ \text{odd} }}^{u+v-1} \slirr{m}^{k+1} \otimes \vir_{m, 1}^{k'}, \qquad k' = \frac{k+3}{k+2} -2. 
\]
The Urod-Theorem, that is the statement that the twisted  quantum Hamiltonian reduction functor of \cite{AF} commutes with tensoring with integrable representations 
\cite[Thm.8.7]{ACF} gives 
\begin{equation}
\vir_{r, s}^k \otimes \slirr{a}^1\cong \bigoplus_{\substack{m =  1 \\  m +r + s + a \ \text{even}  }}^{u+v-1} \vir^{k+1}_{m, s} \otimes \vir_{m, r}^{\frac{k}{k+1}}
 \end{equation}
 for $r =1, \dots, u-1, s=1, \dots, v-1, a =1, 2$. 
 Note that the conformal  vector on the lattice vertex operator algebra is unusual, so that this statement extends to the character identity if we use the usual conformal vectors. 
\begin{equation}
\text{ch}[ \vir_{r, s}^k](q)  \ \text{ch}[\slirr{a}^1](1, q)  =  \sum_{\substack{m =  1 \\  m +r + s + a \ \text{odd}  }}^{u+v-1} \text{ch}[ \vir^{k+1}_{m, s}](q)  \ \text{ch}[\vir_{m, r}^{\frac{k}{k+1}}](q)
  \end{equation}
This gives us branching rules in general. 
Firstly, by \cite[Thm.7.2]{Ad1} we have 
\begin{equation}
\text{ch}[\slrel{\lambda}{r,s}^k ](z, q) = z^{-k +2\lambda}\text{ch}[ \vir_{r, s}^k](q)  \frac{\delta(z^2)}{\eta(q)^2}
\end{equation}
Using that 
\[
 \text{ch}[ \slirr{a}^1] (z, q) \ \delta(z^2) =  z^a \text{ch}[ \slirr{a}^1] (1, q) \ \delta(z^2) 
\]
we get 
\begin{equation}
\begin{split}
\text{ch}[\slrel{\lambda}{r,s}^k](z, q)  \ \text{ch}[\slirr{a}^1](z, q)  &= 
z^{-k +2\lambda} \text{ch}[ \vir_{r, s}^k](q)  \frac{\delta(z^2)}{\eta(q)^2}  \ \text{ch}[\slirr{a}^1](z, q) \\
&= z^{-k +2\lambda +a} \text{ch}[ \vir_{r, s}^k](q) \frac{\delta(z^2)}{\eta(q)^2}  \ \text{ch}[\slirr{a}^1](1, q) \\
&= z^{-k-1 +2\lambda +a+1}  \frac{\delta(z^2)}{\eta(q)^2}   \sum_{\substack{m =  1 \\  m +r + s + a \ \text{odd}  }}^{u+v-1} \text{ch}[ \vir^{k+1}_{m, s}](q)  \ \text{ch}[\vir_{m, r}^{k'}](q) \\
&=  \sum_{\substack{m =  1 \\  m +r + s + a \ \text{odd}  }}^{u+v-1} \text{ch}[\slrel{\lambda + \frac{a+1}{2}}{m, s}^{k+1}](z, q)   \ \text{ch}[\vir_{m, r}^{k'}](q) .
\end{split}
\end{equation}
Spectral flow changes the character as \cite[(4.13)]{CR1}
\[
\ch[\sigma^\ell(M)](z, q) = z^{k \ell } q^{\frac{k \ell^2}{4}} \ch[M](z q^{\frac{\ell}{2}}, q)
\]
and so (here for an integer $n$, we set $\underline{n}$ to be one if $n$ is odd and to be two if $n$ is even) 
\begin{equation}
\begin{split}
\ch[\sigma^\ell(\slrel{\lambda}{r,s}^k)](z, q)]  &\ch[\slirr{\underline{a+\ell}}^1](z, q) = \ch[\sigma^\ell(\slrel{\lambda}{r,s}^k](z, q)]  \ch[\sigma^\ell(\slirr{a}^1)] (z, q)\\
&=  \sum_{\substack{m =  1 \\  m +r + s + a \ \text{odd}  }}^{u+v-1} \text{ch}[\sigma^\ell(\slrel{\lambda + \frac{a+1}{2}}{m, s}^{k+1})](z, q)   \ \text{ch}[\vir_{m, r}^{k'}](q) .
\end{split}
\end{equation}
An immediate consequence is
\begin{corollary}\label{cor:branching}
The following branching rules hold
\begin{align}
\sigma^\ell(\slrel{\lambda}{r,s}^k) \otimes \slirr{\underline{a+\ell}}^1 &=  \label{l1}
 \bigoplus_{\substack{m =  1 \\  m +r + s + a \ \text{odd}  }}^{u+v-1} \sigma^\ell(\slrel{\lambda + \frac{a+1}{2}}{m, s}^{k+1}) \otimes \vir_{m, r}^{k'} \\  \label{l2}
 \sfmod{\ell}{\slindrel{r,s}^{\pm, k}}\otimes \slirr{\underline{a+\ell}}^1 &= 
 \bigoplus_{\substack{m =  1 \\  m +r + s + a \ \text{odd}  }}^{u+v-1} \sfmod{\ell}{\slindrel{m, s}^{\pm, k+1}} \otimes \vir_{m, r}^{k'}\\  \label{l3}
 \sfmod{\ell}{\sldis{r,s}^{\pm, k}}\otimes \slirr{\underline{a+\ell}}^1 &= 
 \bigoplus_{\substack{m =  1 \\  m +r + s + a \ \text{odd}  }}^{u+v-1} \sfmod{\ell}{\sldis{m, s}^{\pm, k+1}} \otimes \vir_{m, r}^{k'} \\  \label{l4}
  \sfmod{\ell}{\slproj{r,s}^k}\otimes \slirr{\underline{a+\ell}}^1 &= 
 \bigoplus_{\substack{m =  1 \\  m +r + s + a \ \text{odd}  }}^{u+v-1} \sfmod{\ell}{\slproj{m, s}^{k+1}} \otimes \vir_{m, r}^{k'}.
\end{align}
for $\lambda \neq \lambda_{r, s}, \lambda_{u-r, v-s}$
\end{corollary}
\begin{proof}
The first equation \eqref{l1} follows from the character identity as simple modules in $\sWtsl{k}$ have unique character. 
For the second case \eqref{l2} we notice that $\sfmod{\ell}{\slindrel{r,s}^+}, \sfmod{\ell}{\slindrel{r,s}^-}$ and $\sfmod{\ell}{\sldis{r,s}^+} \oplus \sfmod{\ell}{\sldis{u-r,v-s}^-}$
have all the same character. But on $\sfmod{\ell}{\slindrel{r,s}^+}$  $\sigma^\ell(e_0)$ acts locally nilpotently while  $\sigma^\ell(f_0)$ doesnot. While for $\sfmod{\ell}{\slindrel{r,s}^-}$  $\sigma^\ell(f_0)$ acts locally nilpotently while  $\sigma^\ell(e_0)$ doesnot and on $\sfmod{\ell}{\sldis{r,s}^+} \oplus \sfmod{\ell}{\sldis{u-r,v-s}^-}$ none of them acts locally nilpotently. 
Since $\slirr{\underline{a+\ell}}^1 $ is integrable all $e_n, f_n$ act locally nilpotently and hence the action on the tensor product is locally nilpotent if and only if it is so on the first factor. 
The third equation \eqref{l3} is the restriction of the second one to the submodule on which $\sigma^\ell(e_0)$ acts locally nilpotently in the case of $ \sfmod{\ell}{\sldis{r,s}^+}$ and  $\sigma^\ell(f_0)$ acts locally nilpotently in the case of $ \sfmod{\ell}{\sldis{r,s}^-}$.

We turn to the last equation \eqref{l4}. 
Firstly Let $W\subset V$ be vertex operator algebras. An exact sequence of $V$-modules is in particular an exact sequence of $W$-modules and so
the decomposition \eqref{l2} together with \eqref{ses_proj} implies that
\[
\sfmod{\ell}{\slproj{r,s}^k}\otimes \slirr{\underline{a+\ell}}^1 = 
 \bigoplus_{\substack{m =  1 \\  m +r + s + a \ \text{odd}  }}^{u+v-1}   X_{\ell, m, s}^{k+1} \otimes \vir_{m, r}^{k'}
\]
for certain modules $X_{\ell, m, s}^{k+1}$ satisfying
\begin{equation}\label{ses_proj2}
\begin{split}
&\dses {\sfmod{\ell}{\slindrel{m,s}^{+,k+1}}} {X_{\ell, m, s}^{k+1}}    {\sfmod{\ell+1}{\slindrel{m,s+1}^{+,k+1}}},\\
&\dses {\sfmod{\ell}{\slindrel{m,v-1}^{+,k+1}}} {X_{\ell, m, v-1}^{k+1}}    {\sfmod{\ell+2}{\slindrel{u-m,1}^{+,k+1}}},\\
%&\dses {\sfmod{\ell}{\slindrel{m,s}^-}} {\sfmod{\ell}{\slproj{r, s}}}    {\sfmod{\ell-1}{\slindrel{m,s+1}^-}} ,\\
%&\dses {\sfmod{\ell}{\slindrel{m,v-1}^-}} {\sfmod{\ell}{\slproj{r, v-1}}}    {\sfmod{\ell-2}{\slindrel{u-m,1}^-}} ,
\end{split}
\end{equation}
for $\ell \in\ZZ, m = 1, \dots,  u+v-1, s = 1, \dots, v-2$. 	

The decomposition \eqref{l4} follows, since $ \sfmod{\ell}{\slproj{r,s}}$ has a nilpotent endomorphism $x$ mapping its top onto its socle $\sfmod{\ell}{\sldis{r,s}}$. Thus by \eqref{l3} the restriction of $x$ to $X_{\ell, m, s}^{k+1} \otimes \vir_{m, r}^{k'}$ is non-zero, but the only module that satisfies \eqref{ses_proj2} and that has a non-zero nilpotent endomorphism is $\sfmod{\ell}{\slproj{m, s}^{k+1}}$.
\end{proof}
 Fix $\ell \in \mathbb Z$ and $ s \in \{ 1, \dots, v-1 \}$. Let $P_r^k$ be either $P_r^k = \sfmod{\ell}{\slproj{r,s}^k}$ or $P_r^k = \sigma^\ell(\slrel{r-1+\lambda}{r,s}^{k}) $ also let $a =0$ if $r+s$ is even and $a=1$ if $r+s$ is odd and choose $\lambda$  such that $P_r^k$ is projective. Then we can compactly write the corollary for these modules as
 \[
  P^k_r \otimes \slirr{\underline{a+\ell}}^1 = 
 \bigoplus_{\substack{m =  1 \\  m \ \text{odd}  }}^{u+v-1}  P^{k+1}_m  \otimes \vir_{m, r}^{k'}
\]
Denote by $X_r^k$ the simple quotient of $P_r^k$, then 
\[
  X^k_r \otimes \slirr{\underline{a+\ell}}^1 = 
 \bigoplus_{\substack{m =  1 \\  m \ \text{odd}  }}^{u+v-1}  X^{k+1}_m  \otimes \vir_{m, r}^{k'}.
\]
The $P^{k+1}_m$ and $X^{k+1}_m$ can then be read off from the previous corollary. They are independent of $r$.
\begin{corollary}
Let $M^k$ be an indecomposable module and
 $\mathcal Y$ be a surjective intertwining operator of type ${M ^k\choose \slirr{r}^k \, P^k_1}$. Then any simple quotient of $M^k$ is isomorphic to $X_r^k$. 
\end{corollary}
\begin{proof}
Let $b = r \mod 2$, so that 
 \[
  \slirr{r}^k \otimes \slirr{\underline{b}}^1 \cong \bigoplus_{\substack{m =  1 \\  m \ \text{odd}  }}^{u+v-1} \slirr{m}^{k+1} \otimes \vir_{m, r}^{k'}.
 \]
Let $\mathcal Y^1$ be a non-zero intertwiner of type  ${  \slirr{\underline{b+a + \ell +1}}^1  \choose \slirr{\underline{b}}^1 \,  \slirr{\underline{a+\ell}}^k }$ and consider $\widetilde{\mathcal Y} = \mathcal Y \otimes \mathcal Y^1$.  Let
$w_1$ be a nonzero vector in  $ \vir_{1, r}^{k'}$ and $w_2$ a non-zero generating vector of $P_1^{k+1}$ .
Then $ \one \otimes w_1$ generates $\slirr{r}^k \otimes \slirr{\underline{b}}^1 $ and $w_2 \otimes \one$ generates $P_1^k \otimes \slirr{\underline{a+\ell}}^1$ (since by the previous corollary no proper submodule of  $P_1^k \otimes  \slirr{\underline{a+\ell}}^1$ contains $P_1^{k+1} \otimes \vir_{1, 1}^{k'}$) . If $\widetilde{\mathcal Y}(\one \otimes w_1, z)(w_2 \otimes \one) =0$, then $\widetilde{\mathcal Y} = 0$ by Proposition \ref{prop:intertwiner}, which can't be. 
Let $X^k$ be a simple quotient of $M^k$ and $\pi$ the projection of $M^k$ onto $X^k$. 
We set 
$\mathcal Y_X :=\pi \circ {\mathcal Y}$ and  $\widetilde{\mathcal Y}_X := \mathcal Y_X \otimes \mathcal Y^1$ and by Proposition \ref{prop:intertwiner} we have that $\widetilde{\mathcal Y}_X(\one \otimes w_1, z)(w_2 \otimes \one) \neq 0$ as well.
But $\widetilde{\mathcal Y}_X$ is in particular an intertwining operator of the subvertex operator algebra $\slirr{1}^{k+1} \otimes \vir_{1, 1}^{k'}$ and so  there exists an $\slirr{1}^{k+1} \otimes \vir_{1, 1}^{k'}$ intertwining operator of type ${X^k \otimes \slirr{\underline{b+a + \ell +1}}^1 \choose \slirr{1}^{k+1} \otimes \vir_{1, r}^{k'} \, X_1^{k+1} \otimes \vir_{1, 1}^{k'}}$. The fusion product is
\[
(\slirr{1}^{k+1} \otimes \vir_{1, r}^{k'}) \boxtimes (X^{k+1}_1 \otimes \vir_{1, 1}^{k'}) = (\slirr{1}^{k+1} \boxtimes X^{k+1}_1) \otimes (\vir_{1, r}^{k'} \boxtimes \vir_{1, 1}^{k'}) =  X_1^{k+1} \otimes \vir_{1, r}^{k'}
\]
 and so $X_1^{k+1} \otimes \vir_{1, r}^{k'}$ must be a quotient (as a $\slirr{1}^{k+1} \otimes \vir_{1, 1}^{k'}$-module) of $X^k \otimes   \slirr{\underline{b+a + \ell +1}}^1$. By the previous corollary the only possibility is $X^k \cong X^k_r$. 
\end{proof}
We will prove the existence of vertex tensor category and this implies the following fusion rules. 
\begin{theorem}\label{thm:fusion}
Let $k$ be admissible and assume that $\sWtsl{k}$ is a vertex tensor category. Then the following fusion rules hold
\begin{equation}
\begin{split}
 \slirr{r}^k \boxtimes  \sigma^{\ell}(\slirr{r'}^k) &\cong  \bigoplus_{r''=1}^{u-1} N^{u\ \ r''}_{r, r'}  \sigma^{\ell}(\slirr{r''}^k) \\
 \slirr{r}^k \boxtimes \sigma^\ell(\sldis{r', s}^\pm) &\cong  \bigoplus_{r''=1}^{u-1} N^{u\ \ r''}_{r, r'}  \sigma^\ell(\sldis{r'',s}^\pm) \\
 \slirr{r}^k \boxtimes  \sigma^\ell(\slrel{\lambda}{r',s}) &\cong  \bigoplus_{r''=1}^{u-1} N^{u\ \ r''}_{r, r'}  \sigma^\ell(\slrel{r-1+\lambda}{r'',s}) \\
 \slirr{r}^k \boxtimes \sfmod{\ell}{\slproj{r',s}^k} &\cong  \bigoplus_{r''=1}^{u-1} N^{u\ \ r''}_{r, r'} \sfmod{\ell}{\slproj{r'',s}^k} \\
  \end{split}
\end{equation}
for all $1 \leq r, r' \leq u-1, 1 \leq s \leq v-1$ and $\lambda \in \mathbb C$. 
\end{theorem}
\begin{proof}
The fusion rules
\[
\slirr{r}^k \boxtimes  \slirr{r'}^k \cong  \bigoplus_{r''=1}^{u-1} N^{u\ \ r''}_{r, r'}  \slirr{r''}^k
\]
as well as rigidity of the $\slirr{r}^k$ is proven in \cite{CHY}. \\[-3mm]

%\noindent{\br Case 1:} $r'=1$

We use the notation of the previous Corollary and consider first the case $r'=1$ and $P^k_1$ projective. By the previous Corollary and since the tensor product of a rigid and a projective module is always projective the only possibility is that the projective cover of $X_r^k$ appears with a certain multiplicity $n_r$. 
\[
\slirr{r}^k \boxtimes P^k_1 = n_r P^k_r.
\] 
This multiplicity can't be zero as the tensor product with a module that has a dual can never vanish. 
We claim that $n_r =1$ and this follows by induction for $r$. The case $r=1$ is clear since $\slirr{1}^k$ is the tensor identity. 
Next consider $r>1$ (and $r \neq u-1$), then by associativity, the fusion rules of ordinary modules and by the induction hypothesis we get
\[
n_{r+1}P^k_{r+1} + P^k_{r-1} \cong (\slirr{r+1}^k \oplus  \slirr{r-1}^k) \boxtimes P^k_1 
\cong (\slirr{2}^k \boxtimes \slirr{r}^k) \boxtimes P^k_1 
\cong n_r (\slirr{2}^k \boxtimes P^k_r).
\] 
$P^k_{r+1}$ and $P^k_{r-1}$ are inequivalent indecomposable projective modules, while $\slirr{2}^k \boxtimes P^k_r$ is projective and clearly the only possibility for this isomorphism to be true is that $n_r = 1$. 
The case $r=u-1$ follows immediately from $\slirr{u-1}^k \boxtimes \slirr{u-1}^k \cong \slirr{1}^k$.

Since $\slirr{r}^k$ is rigid, the tensor product with $\slirr{r}^k$ is exact. Since the number of composition factors of $P^k_1$ and $P^k_r$ coincide and since the tensor product with 
$\slirr{r}^k$ cannot be zero clearly the only possibility is that the tensor product of $\slirr{r}^k$  with any simple composition factor of $P^k_1$ is a single simple compsoition factor of $P^k_r$. The only possibility that is consistent with composition series are 
\[
\slirr{r}^k \boxtimes  \sigma^{\ell}(\slirr{1}^k) \cong    \sigma^{\ell}(\slirr{r}^k), \qquad
 \slirr{r}^k \boxtimes \sigma^\ell(\sldis{1, s}^\pm) \cong \sigma^\ell(\sldis{r, s}^\pm).
\]
Finally 
\[
\slirr{r}^k \boxtimes  \sigma^\ell(\slrel{\lambda}{1,s}) \cong \sigma^\ell(\slrel{r-1+\lambda}{r,s})
\]
with $\lambda = \lambda_{r,s}, \lambda_{u-r,v-s} \pmod{\rlat}$ follows from exactness of $\slirr{r}^k \boxtimes \, \bullet \, $.

The claimed fusion rules follow immediately via:
Let $M_1, \dots, M_{u-1}$ be a set of objects in $\sWtsl{k}$  with the property that
\[
\slirr{r}^k \boxtimes M_1 \cong M_r
\]
for all $r=1, \dots, u-1$
then by associativity
\begin{equation} \nonumber
\begin{split}
\slirr{r}^k \boxtimes M_{r'} &\cong \slirr{r}^k \boxtimes (\slirr{r'}^k \boxtimes M_1) 
 \cong (\slirr{r}^k \boxtimes \slirr{r'}^k) \boxtimes M_1 
 \cong \bigoplus_{r''=1}^{u-1} N^{u\ \ r''}_{r, r'}  \slirr{r''}^k \boxtimes M_1 \\
 &\cong \bigoplus_{r''=1}^{u-1} N^{u\ \ r''}_{r, r'}  M_{r''}.
\end{split}
\end{equation}
\end{proof}

\section{The $N=2$ super Virasoro algebra}\label{N=2}

The $N=2$ super Virasoro algebra is the principal $\W$-algebra of $\ssl$.  We denote the universal principal $\W$-algebra of $\ssl$ at level $k$ by $\W^k(\ssl)$ and its simple quotient by $\W_k(\ssl)$. 
It is strongly generated by even fields $L(z), J(z)$ of conformal weight two and one and odd fields $G^\pm(z)$ of conformal weight $3/2$. With the convention
\[
L(z) = \sum_{n \in \mathbb Z} L_n z^{-n-2}, \qquad 
J(z) = \sum_{n \in \mathbb Z} X_n z^{-n-1}, \qquad 
G^\pm(z) = \sum_{n \in \mathbb Z+\frac{1}{2}} G^\pm_n z^{-n-\frac{3}{2}}, \qquad 
\]
the commutation relations of the modes are
\begin{equation}
\begin{split}
[L_m, L_n] &= (m-n) L_{m+n} + \frac{m(m^2-1)}{12} \delta_{m+n, 0} c, \\
[L_m, J_n] &= -n J_{m+n},  \\
[J_m, J_n] &= m \delta_{m+n, 0} \frac{c}{3}, \\
 [L_m, G^\pm_s ] &= \left(\frac{m}{2} -s \right) G^\pm_{m+s},\\
 [J_m, G^\pm_s] &= \pm G^\pm_{m+s}, \\
  [ G^+_r, G^-_s] &= 2L_{r+s} + (r-s) J_{r+s} + \frac{4r^2-1}{12} \delta_{r+s, 0} c.
\end{split}
\end{equation}
These generate the $N=2$ superconformal Lie algebra at central charge $c$. The central charge is $c = \frac{3 \ell}{\ell + 2}$ and $\ell$ is related to $k$ via the Feigin-Frenkel type relation $(\ell+2)(k+1) =1$. 
This relation comes from the Kazama-Suzuki realization \cite{KS}
\[
\W_k(\ssl) \cong \text{Com}(\pi, L_\ell(\sltwo) \otimes \cF)
\]
where $\cF$ is the vertex operator algebra of two free fermions $\psi^+(z), \psi^-(z)$ with non-zero  OPE
\[
\psi^+(z)\psi^-(w) = (z-w)^{-1}.
\]
Here $\pi$ is the rank one Heisenberg vertex operator algebra generated by $h+ 2:\psi^+\psi^-:$ and $\ell$ needs to be non-critical.   

Let $\BssWtsl{k}$ respectively $\ssWtsl{k}$ be the category of grading-restricted generalized $\W^k(\ssl)$ respectively $\W_k(\ssl)$-weight modules, i.e.  lower-bounded generalized $\W^k(\ssl)$ respectively $\W_k(\ssl)$-modules with finite-dimensional conformal weight spaces and such that  the Cartan subalgebra $\mathfrak h = \mathbb C X_0$ acts semisimply. 
Examples of objects in $\BssWtsl{k}$ are Verma-modules $\mathcal V_{\Delta, \lambda}$ of conformal weight $\Delta$ and weight $\lambda$. These are generated by a highest-weight vector $| \Delta, \lambda \rangle$, satisfying
\[
L_0 | \Delta, \lambda \rangle = \Delta | \Delta, \lambda \rangle, \quad J_0 | \Delta, \lambda \rangle = \lambda | \Delta, \lambda \rangle, \quad 
X_n | \Delta, \lambda \rangle = 0, \quad  n>0, \quad X \in \{ G^\pm, L, J\}.
\]
The $N=2$ superconformal Lie algebra at central charge $c$ has an involution $\tau$ acting as 
\[
\tau(G^\pm_r) = G^\mp_r, \qquad \tau(L_n) = L_n, \qquad \tau(J_m) = - J_m. 
\]
For a module $M$ in $\ssWtsl{k}$, we denote by $M^\tau$ the module $M$ twisted by $\tau$, in particular 
\[
\mathcal V_{\Delta, \lambda}^\tau \cong. \mathcal V_{\Delta, -\lambda}.
\]

We return to the Kazama-Suzuki coset relation.
This relation is excellent as it allows to efficiently study the representation theory of $\W_k(\ssl)$ when $\ell$ is an admissible level for $\sltwo$, see \cite{CLRW}.

It in particular means \cite{CKLR} that there is a set of simple $\W_k(\ssl)$-modules $\W_n$ for $n \in 2\mathbb Z$, such that
\begin{equation}\label{eq:ext}
 L_\ell(\sltwo) \otimes \cF \cong \bigoplus_{n \in 2\mathbb Z} \W_n \otimes \pi_n
\end{equation}
as  $\W_k(\ssl) \otimes \pi$-modules 
with $\pi_n$ the Fock-module of highest-weight $n$. 

The correspondence between modules is best phrazed in terms of relative semi-infinite Lie algebra cohomology of $\mathfrak{gl}_1((z))$ relative to $\mathfrak{gl}_1$ with coefficient in 
$\pi^+_\lambda \otimes \pi^-_\mu$. Here $\pi^\pm$ are rank one Heisenberg vertex operator algebras of some level $\pm\kappa$ \cite{CGNS}. Denote this cohomology simply by $H$, then 
\[
H^p(\pi^+_\lambda \otimes \pi^-_\mu) = \delta_{p, 0} \delta_{\lambda, -\mu} \mathbb C.
\]
Let $\pi^+ = \pi$ be the  rank one Heisenberg vertex operator algebra generated by $h+ 2:\psi^+\psi^-:$, it has level $\kappa = 2k+2$, and let $\pi^-$ be another rank one Heisenberg vertex operator algebra of level $-\kappa$.
Then $H^0$ gives a family of functors $H_\lambda$ from $\sWtsl{\ell}$  to the category $\ssWtsl{k}$ of weight modules of $\W_k(\ssl)$, 
\[
H_\lambda: M \mapsto H^0(M \otimes \cF \otimes \pi^-_\lambda).
\]
This functor gives a block-wise equivalence of abelian categories, see the introduction of \cite{CGNS} for the precise statement. In particular for any simple $\W_k(\ssl)$-module $X$
there exists a pair $(M, \pi^-_\lambda)$ with $M$ a simple $L_k(\sltwo)$-module, such that
\[
X \cong H^0(M \otimes \cF \otimes \pi^-_\lambda).
\]
Note, that the pair $(M, \pi^-_\lambda)$ is not unique with this property. 
Due to Theorem \ref{classification} it follows
\begin{corollary} \textup{(Corollary of \cite{CGNS})} \label{cor:fl}
Let $\ell$ be an admissible level for $\sltwo$ and $k$ determined by $(\ell+2)(k+1) =1$. Then $\ssWtsl{k}$ is of finite length.
\end{corollary}

\section{Free field realizations}\label{ff}

We need the Adamovic free field realization \cite{Ad1}.
Consider the lattice $L = {\Bbb Z} c + {\Bbb Z} d$ such that
$$ \langle c, c \rangle = \langle d, d \rangle=0, \ \langle c, d \rangle =2, $$ and lattice vertex algebra $V_L = \pi^{c,d} \otimes {\C}[L]$,  where
$\pi^{c,d}$ is rank two Heisenberg vertex algebra, and ${\C}[L]$ group algebra of lattice $L$. $\pi^{c, d}$ is generated by two Heisenberg fields $c(z), d(z)$ with OPEs
\[
c(z)c(w) = 0 = d(z)d(w) \qquad\text{and} \qquad c(z)d(w) = 2(z-w)^{-2}.
\]
Consider the vertex algebra $\Pi(0) = \pi^{c,d} \otimes {\C}[{\Bbb Z}c]\subset V_L$.
Let $c_0, d_0$ be the zero-modes of $c(z), d(z)$ and let $\pi^{c, d}_{\lambda, \mu}$ be the Fock module of $(c_0, d_0)$ weight $(\lambda, \mu)$. Then 
\[
\Pi(0) = \bigoplus_{n \in \mathbb Z } \pi^{c, d}_{0, 2n}
 \]
 as $\pi^{c, d}$-module. This means that $c_0$ commutes with $\Pi(0)$ in the sense that $[c_0, Y(v, z)] =0$ for any $v \in \Pi(0)$. Hence $\Pi(0) \otimes_{\mathbb C} \mathbb C[c_0]$ is a $\Pi(0)$-module as well. In particular there is a self-extension
\[
\widetilde \Pi(0) := \Pi(0) \otimes_{\mathbb C} \mathbb C[c_0]/c_0^2
\]
on which $c_0$ acts nilpotently, that is 
\[
0 \rightarrow {\Pi(0)} \rightarrow {\widetilde \Pi(0)}\xrightarrow{c_0}{\Pi(0)} \rightarrow 0
\]
Let $\mu $ be defined by $2\mu =\frac{\ell}{2}c +d$.
The simple modules of $\Pi(0)$ are $\Pi_m(\lambda) := \Pi(0) e^{m\mu +\lambda c} $ for $m \in \mathbb Z$ and $\lambda \in \mathbb C$.  $c(0)$ acts on $\Pi_m(\lambda)$ by multiplication with $m$.  Again we have a self-extension 
\[
\widetilde \Pi_m(\lambda) := \Pi_m(\lambda) \otimes_{\mathbb C} \mathbb C[c_0]/(c_0 - m \text{Id})^2
\]
on which $c_0-m \mathrm{Id}$ acts nilpotently, that is 
\begin{equation}\label{pi_ext}
0 \rightarrow {\Pi_m(\lambda)} \rightarrow {\widetilde \Pi_m(\lambda)}\xrightarrow{c_0 - m\mathrm{Id}}{\Pi_m(\lambda)} \rightarrow 0
\end{equation}

\begin{theorem} \textup{\cite{Ad1}}
There is a conformal embedding
\[
V^\ell(\sltwo) \hookrightarrow \vir^\ell \otimes \Pi(0)\]
and  for non-integral admissible level $\ell$,  
\[
L_\ell(\sltwo) \hookrightarrow \vir_\ell \otimes \Pi(0)
\]
mapping $h$ to $\frac{\ell}{2}c +d$ and $c_\ell = 1 - 6 \frac{(\ell+1)^2}{\ell+2}$.
\end{theorem}
In section 7 of \cite{Ad1} simple modules of $\vir_\ell \otimes \Pi(0)$ are identified with modules in $\sWtsl{\ell}$, see in particular Proposition 7.1 and Remark 6 there. Comparing with the list of modules and the complete list of inequivalent simple objects in $\sWtsl{\ell}$ we see
\begin{proposition}\textup{\cite[Section 7]{Ad1} }\label{prop:ada}
Let $\ell$ be admissible and not integral. 
Let 
\[
I := \left\{ (r, s,  m) \ | \ r = 1,\dots,u-1, \ s = 1,\dots,v-1 \ \text{and} \ m \in \mathbb Z\right\}
\]
then the two sets of objects in $\sWtsl{\ell}$ coincide
\begin{equation}\nonumber
\begin{split}
		S_1 &:= \left\{ \sigma^m(\slrel{\lambda}{r,s}) \ | \  (r, s, m) \in I  \ \text{and} \ \lambda \in \alg{h}^* \ \text{with} \ \lambda \neq \lambda_{r,s}, \lambda_{u-r,v-s} \pmod{\rlat} \ \right\}\\ & \qquad   \cup \ 	\left\{ {\sfmod{m }{\slindrel{r,s}^-}} \ | \  (r, s, m) \in I  \right\} \\
		S_2&:= \left\{  	 \Pi_m(\lambda) \otimes \vir^\ell_{r, s} \ |  \  (r, s, m) \in I  \ \text{and} \ \lambda \in \alg{h}^* \pmod{\rlat} \ \right\}.
		\end{split}
	\end{equation}
In particular let
 $M$ be a simple object in $\sWtsl{\ell}$, then there exists a simple $\vir_\ell \otimes \Pi(0)$-module $\widetilde M$, such that its simple quotient as an object in $\sWtsl{\ell}$ is isomorphic to $M$.
\end{proposition}
Tensoring the self-extensions \eqref{pi_ext} with $\vir^\ell_{r, s}$ gives the self-extensions 
\begin{equation}
0 \rightarrow {\Pi_m(\lambda)\otimes \vir^\ell_{r, s}} \rightarrow {\widetilde \Pi_m(\lambda)\otimes \vir^\ell_{r, s}}\xrightarrow{c_0 - m\mathrm{Id}}{\Pi_m(\lambda)\otimes \vir^\ell_{r, s}} \rightarrow 0
\end{equation}
These are in particular self-extensions as modules for $L_\ell(\sltwo)$. However they are not in $\sWtsl{\ell}$, since $c_0$ and hence $h_0=\frac{\ell}{2}c(0) +d(0)$ doesn't act semisimply.
We record this as a corollary
\begin{corollary}
For every object $M$ in $S_1$ there exists a self-extension $\widetilde M$, that is a non-split exact sequence
\[
0 \rightarrow M \rightarrow \widetilde M \rightarrow M \rightarrow 0.
\]
$\widetilde M$ is a module for $L_\ell(\sltwo)$, but it is not an object of $\sWtsl{\ell}$, since $h_0=\frac{\ell}{2}c_0 + d_0$ doesn't act semisimply.

\end{corollary} 
This Corollary is needed for the existence of the indecomposable modules $\sfmod{\ell}{\slproj{r,s}^k}$ \cite[Appendix A]{ACK}.

Consider $L_\ell(\sltwo) \otimes \cF$, we denote the fermions by $\widetilde \psi^\pm$  and let us bosonize the fermions, that is 
\[
\cF \cong \bigoplus_{m \in \mathbb Z} \pi^g_m
\]
with $g$ a rank one Heisenberg vertex operator algebra of level one, identified with $g = :\widetilde \psi^+ \widetilde \psi^-:$. 
We choose a new basis 
\[
Z:= \frac{\ell}{2}c + d +2g, \qquad
Y := g- c, \qquad
X := -\frac{4}{\ell}g -\frac{2}{\ell} d + \left(1 + \frac{4}{\ell}\right)c. 
\]
This basis is chosen such that these three fields are orthogonal on each other and such that $\pi^Z$ is the image of $\pi$ under the embedding of $L_\ell(\sltwo) \otimes \cF \hookrightarrow \vir_{c_\ell} \otimes \Pi(0) \otimes \cF$. Note that $Y^2=1$ and $X^2 = -8\frac{\ell+2}{\ell^2}$.
We decompose $ \Pi(0) \otimes \cF$ into Fock modules
\begin{equation}\nonumber
\Pi(0) \otimes \cF \cong \bigoplus_{\substack{ n \in 2 \mathbb Z \\ m \in \mathbb Z}} \pi^c_0 \otimes \pi^d_n \otimes \pi^g_m 
\cong \bigoplus_{\substack{ n \in 2 \mathbb Z \\ m \in \mathbb Z}} \pi^Z_{n+2m} \otimes \pi^Y_m \otimes \pi^X_{ -\frac{4}{\ell}m- \frac{2}{\ell}n} 
\end{equation}
so that 
\[
\text{Com}(\pi^Z, \Pi(0) \otimes \cF) \cong \bigoplus_{m \in \mathbb Z}  \pi^Y_m \otimes \pi^X
\cong \cF \otimes \pi^X.
\]
Here we use that $ \bigoplus_{m \in \mathbb Z}  \pi^Y_m \cong \cF$ and let us 
set $\psi^+ := e^{-Y}, \psi^-:= e^Y$.
Note that  with replacing $\pi$ by $\pi^Z$ the coset is also realized as semiinfinite cohomology, $H^0( \Pi(0) \otimes \cF \otimes \pi^-) = \text{Com}(\pi^Z \Pi(0) \otimes \cF)$. 
In particular the embeddings
$V^\ell(\sltwo) \hookrightarrow \vir^\ell \otimes \Pi(0)$ and 
 $L_\ell(\sltwo) \hookrightarrow \vir_\ell \otimes \Pi(0)$ induce
\begin{equation}\nonumber
\begin{split}
\W^k(\ssl)  &\cong \text{Com}(\pi, V^\ell(\sltwo) \otimes \cF) \hookrightarrow \text{Com}(\pi^Z,  \vir^{c_\ell} \otimes \Pi(0) \otimes \cF) \cong  \vir^\ell \otimes \cF \otimes \pi^X\\
\W_k(\ssl)  &\cong \text{Com}(\pi, L_\ell(\sltwo) \otimes \cF) \hookrightarrow \text{Com}(\pi^Z,  \vir_{c_\ell} \otimes \Pi(0) \otimes \cF) \cong  \vir_\ell \otimes \cF \otimes \pi^X.
\end{split}
\end{equation}
Note that in the free field realization the field $e(z)$ of $V^k(\sltwo)$ corresponds to $e^c$ and in the coset there is the field $G^+=:e(z) c(z): =e^{c-g}$ which after fermionization is identified with one of the fermions, $\psi^+$, that is $G^+ = \psi^+$ in the free field realization. The field $e(z), h(z), f(z)$ are given explicitely in \cite[Eq. 17]{Ad1} and the Virasoro field is given in  \cite[Eq. 19]{Ad1}.
 It is a quick computation that then
\begin{equation}\label{Gminus}
\begin{split}
J &= -\frac{\ell}{2(\ell+2)} X + :\psi^+ \psi^-:\\
G^+ &= \psi^+ \\
G^- &= \left( (\ell + 2) L - :\left( :\psi^+ \psi^-: - \frac{\ell}{4} X\right)\left( :\psi^+ \psi^-:- \frac{\ell}{4} X\right):\right. \\
&\qquad \left. - (\ell+1) \frac{d}{dz}\left(  :\psi^+ \psi^-: - \frac{\ell}{4} X\right)\right) \psi^- \\
T&= L + L^{\psi} + L^X + \frac{d}{dz} \left( \frac{\ell}{4}X -  :\psi^+ \psi^-: \right) \\
\end{split}
\end{equation}
with $L$ the Virasoro field of $\vir^\ell$ respectively of $\vir_\ell$.
Introduce some notation 
\begin{notation}${}$ \label{notation}

\begin{itemize}
\item Let $A = \{a_1, \dots, a_n\}$, $a_i \in \mathbb Z$ with $a_1\geq a_2 \geq \dots \geq a_n \geq 1$ and set $|A| = n$. 
\item Let $C = \{c_1, \dots, c_m\}$, $c_i\in \mathbb Z$ with $c_1\geq c_2 \geq \dots \geq c_m \geq 1$ and set $|C| = m$.
\item Let $B = \{b_1, \dots, b_r\}$, $b_i\in \frac{1}{2}\mathbb Z$ with $b_1 >  b_2 > \dots > b_r \geq 1/2$ and $b_i \not\in \mathbb Z$ and set $|B| = r$.  
\item Let $\tilde B = \{\tilde b_1, \dots, \tilde b_s\}$, $\tilde b_i \in \frac{1}{2}\mathbb Z$ with $\tilde b_1 >  \tilde b_2 > \dots >\tilde b_s \geq 1/2$ and  $\tilde b_i \not\in \mathbb Z$ set $|\tilde B| = s$.  
\item Set 
\begin{equation}
\begin{split}
T_{-A} &= T_{-a_1} \dots T_{-a_n}, \qquad  L_{-A} = L_{-a_1} \dots L_{-a_n}, \\
G^+_{-B} &= G^+_{-b_1} \dots  G^+_{-b_r},  \qquad  \psi^+_{-B} = \psi^+_{-b_1} \dots  \psi^+_{-b_r},  \\
G^-_{-\tilde B} &= G^-_{-\tilde b_1} \dots  G^-_{-\tilde b_s},\qquad \psi^-_{-\tilde B} = \psi^-_{-\tilde b_1} \dots  \psi^-_{-\tilde b_s}, \\
J_{-C} &= J_{-c_1} \dots  J_{-c_m}, \qquad X_{-C} = X_{-c_1} \dots  X_{-c_m}. 
\end{split}
\end{equation}
\item 
We say that the monomial $L_{-A}X_{-C}\psi^+_{-B} \psi^-_{-\tilde B}$ has $L$-degree $|A|$, $X$-degree $|C|$, et cetera.
\end{itemize}
\end{notation}

\begin{corollary}\label{cor:mod_realized}
Let $\ell \notin \{ 0-2\}$ and $M_h$ be a $\vir^\ell$ module generated by a highest-weight vector $v$ of highest-weight $h \in \mathbb C$, then 
\[
M_h \otimes \cF \otimes \pi^X_\lambda \cong \W^k(\ssl).v
\]
for  $\lambda \in \mathbb C$ and 
$ h \notin \left\{ \frac{1}{\ell+2}  \left( \left( r-1+ \frac{\ell}{4} \lambda \right)^2 - (\ell+1)\left(r-1 + \frac{\ell}{4}\lambda\right) \right) \ \Big|   \ r \in \mathbb Z_{>0} \ \right\}$. 
\end{corollary}
\begin{proof}
$M_h \otimes \cF \otimes \pi^X_\lambda$ is spanned by monomials of the form
$L_{-A}X_{-C}\psi^+_{-B} \psi^-_{-\tilde B} v$.
Consider $\tilde B = \{\tilde b_1, \dots, \tilde b_s\}$ as in Notation \ref{notation}. Set $D = \{\tilde b_1, \tilde b_1 -1, \tilde b_1 -2,  \dots, \frac{1}{2} \}$. 
Then from \eqref{Gminus}
\begin{equation} \label{eq:pr}
\begin{split}
G^-_{-D}v &= \prod_{r=1}^{|D|} p_r(h, \lambda) \psi^-_{-D}v \qquad \text{with} \\ \qquad p_r(h, \lambda) &= (\ell + 2) h - \left( r-1+ \frac{\ell}{4} \lambda \right)^2 - (\ell+1)\left(r-1 + \frac{\ell}{4}\lambda\right).
\end{split}
\end{equation}
Let $E := D \setminus\tilde B$ and ordered as before, i.e. $e_1 > e_2 > \dots$. Set $G^+_E = \prod\limits_{i=1}^{|D| - | \tilde B|}G^+_{e_i}$.
Then 
\[
G^+_E G^-_{-D}v =  \epsilon \prod_{r=1}^{|D|} p_r(h, \lambda) \psi^-_{-\tilde B}v
\]
and some $\epsilon \in \{ \pm 1\}$. 
Assume that none of the $p_r(h, \lambda) $ vanishes,  then any element of the form $\psi^+_{-B} \psi^-_{-\tilde B} v$ is in $ \W^k(\ssl).v$. Set $\kappa = -\frac{\ell}{2 (\ell+2)}$.
By \eqref{Gminus}
$\kappa^{|C|} X_{-C}\psi^+_{-B} \psi^-_{-\tilde B} v$ and $J_{-C}\psi^+_{-B} \psi^-_{-\tilde B} v$ coincide up to sums of monomials of $X$-degree strictly less then $|C|$ (and $L$-degree zero). Thus by induction for $|C|$ we see that any element of the form $X_{-C}\psi^+_{-B} \psi^-_{-\tilde B} v$ is in $ \W^k(\ssl).v$. Using \eqref{Gminus} again and a similar induction for $|A|$ implies that any element of the form $L_{-A}X_{-C}\psi^+_{-B} \psi^-_{-\tilde B} v$  is in $ \W^k(\ssl).v$.
\end{proof}
\begin{remark}\label{rem:generic}
The Corollary says in particular that if $M_h = V(c_\ell, h)$ is a Virasoro Verma module then for generic $\lambda$ the module $ V(c_\ell, h) \otimes \cF \otimes \pi^X_\lambda$ is a Verma module for $\W^k(\ssl)$. The conformal weight of the top level of $ V(c_\ell, h) \otimes \cF \otimes \pi^X_\lambda$ is 
\[
\Delta(h, \lambda) = h - \frac{1}{2} - \frac{\ell^2 \lambda^2 }{16(\ell+2)} - \frac{\ell \lambda}{4}
\]
and the $J_0$-weight is 
\[
\mu(\lambda) = -\frac{\ell \lambda}{2 (\ell+2)} -1.
\]
Both these quantities are read off from \eqref{Gminus}. 
Hence for generic $\lambda$
\[
V(c_\ell, h) \otimes \cF \otimes \pi^X_\lambda \cong \mathcal V_{\Delta(h, \lambda), \mu(\lambda)}.
\]
Set 
\[
\lambda^\tau := -\lambda - 4\frac{(\ell+2)}{\ell}.
\]
Then we compute that
$\Delta(h, \lambda^\tau) = \Delta(h, \lambda)$ and  $\mu(\lambda^\tau) = - \mu(\lambda)$.
Hence for generic $\lambda^\tau$ also 
\[
 \mathcal V_{\Delta(h, \lambda), \mu(\lambda)}^\tau \cong  \mathcal V_{\Delta(h, \lambda), -\mu(\lambda)} \cong  V(c_\ell, h)  \otimes \cF \otimes \pi^X_{\lambda^\tau}.
\]
\end{remark}

\begin{proposition}\label{prop;emb}
Let $\ell$ be admissible and non-integral, $k$ defined by $(\ell+2)(k+1)=1$ and let $N$ be a simple object in  $\ssWtsl{k}$, then there exists a simple $\vir_\ell \otimes \cF \otimes \pi^X$-module $\widetilde N$, such that its simple quotient as an object in  $\ssWtsl{k}$ is isomorphic to $N$.
\end{proposition}
\begin{proof}
Let $N$ be a simple object in   $\ssWtsl{k}$. As discussed at the end of the previous section and thanks to \cite{CGNS} there exists a pair $(M, \pi^-_\lambda)$ with $M$ a simple $L_k(\sltwo)$-module, such that
$N \cong H^0(M \otimes \cF \otimes \pi^-_\lambda)$. 
Let $\widetilde M$ be a simple  $\vir^\ell \otimes \Pi(0)$-module whose simple quotient as an object in $\sWtsl{\ell}$ is isomorphic to $M$, see Proposition \ref{prop:ada}.
Since $H^0( \, \bullet\,  \otimes \cF \otimes \pi^-_\lambda)$ gives a block-wise equivalence of categories the module $\widetilde N = H^0(\widetilde M \otimes \cF \otimes \pi^-_\lambda)$ is a $\vir_\ell \otimes \cF \otimes \pi^X$-module whose simple quotient as an object in  $\ssWtsl{k}$ is $N$. 
$\widetilde N$ is simple as a $\vir_\ell \otimes \cF \otimes \pi^X$-module since the theory of \cite{CGNS} also works for the pair $\vir_\ell \otimes \Pi(0)$ and $\vir_\ell \otimes \cF \otimes \pi^X$.
\end{proof}

\section{$C_1$-cofiniteness in $\ssWtsl{k}$}\label{C1}

The aim of this section is to prove the statement:
Let $\ell$ be an admissible level for $\sltwo$ and $k$ be defined via $(\ell+2)(k+1) =1$. Then every object of  $\ssWtsl{k}$ is $C_1$-cofinite.

We start by introducting $C_1$-cofiniteness and some relevant results.

\subsection{$C_1$-cofiniteness}

We review some properties around $C_1$-cofiniteness. 
The definition is taken from \cite{C-Vir}
\begin{definition}
Let $V$ be a vertex operator superalgebra and let $W$ be a weak $V$-module. Let $C_1(W)$ be the subspace of $W$ spanned by elements of the form $u_{-1}w$, where $u\in V_+:=\coprod_{n\in \ZZ_{>0}} V_{(n)}$ and $w\in W$. We say that $W$ is \textit{$C_1$-cofinite} if the \textit{$C_1$-quotient} $W/C_1(W)$ is finite-dimensional.
\end{definition}
A \textit{generalized $V$-module} is a weak $V$-module $(W, Y_{W})$ with a
$\CC$-grading
\begin{equation} \label{eq:Vmodgrading}
	W=\coprod_{n\in \CC}W_{[n]}
\end{equation}
such that $W_{[n]}$ is a
generalized eigenspace for the operator $L_{0}$ with eigenvalue $n$. Yi-Zhi Huang proved the next Lemma for vertex operator algebras, but the proof for vertex operator superalgebras is the same.
\begin{lemma}[{\cite[Lemma 2.10 and 2.11]{H09}}] \label{c1cofinite}
Let $V$ be a vertex operator superalgebra.
\begin{enumerate}
	\item For a (generalized) $V$-module $W$ and a (generalized) $V$-submodule $U$ of $W$, \label{qu}
	\begin{equation}
		C_1(W/U)=\frac{C_1(W)+U}{U} \quad \text{and so} \quad \frac{W/U}{C_1(W/U)}\cong\frac{W}{C_1(W)+U}.
	\end{equation}
	It follows that $W/U$ is $C_1$-cofinite if $W$ is.
	\item If $W$ is a finite-length $V$-module with $C_1$-cofinite %simple
	composition factors, then $W$ is also $C_1$-cofinite. \label{inclusion}
\end{enumerate}
\end{lemma}
An important property of $C_1$-cofinite modules is that they are closed under fusion.
\begin{theorem}[\cite{Mi}]\label{miyamoto}
Let $W_1, W_2$ be $C_1$-cofinite. If $W_3$ is a lower-bounded generalized module such that there exists a surjective intertwining operator of type $\binom{W_3}{W_1\,W_2}$, then $W_3$ is also $C_1$-cofinite. In particular, the $P(z)$-tensor product $W_1 \boxtimes_{P(z)} W_2$ is $C_1$-cofinite.
\end{theorem}
Miyamoto proved this Theorem for vertex operator algebras in section 2 of \cite{Mi} and we inspect that the proof of the same statement for vertex operator superalgebras is exactly the same.
For the Virasoro algebra $C_1$-cofinite modules are:  
\begin{proposition}\label{noc1} \textup{\cite[Cor. 2.2.7]{C-Vir}}
\leavevmode
\begin{enumerate}
	\item A \hw{} $M(c,0)$-module is $C_1$-cofinite if and only if it is not isomorphic to a Verma module.  In particular, $L(c, h)$ is $C_1$-cofinite if and only if $h \in H_c$. \label{it:c1notverma}
	\item A $C_1$-cofinite \hw{} $M(c,0)$-module
	has finite length. \label{hig}
	\item The composition factors of a $C_1$-cofinite \hw{} $M(c,0)$-module are of the form $L(c,h)$ with $h \in H_c$. \label{it:c1compfacts}
\end{enumerate}
\end{proposition}

\subsection{Li's standard filtration}

Li introduced two filtrations on vertex algebras and their modules \cite{LiCofin}. We need the standard filtration, which is an increasing filtration.
Let $V$ be a vertex operator superalgebra and fix a conformal structure on $V$ and fix a set $\{ X^i | i \in I \}$ of homogeneous strong generators of $V$ and let $\Delta_i$ be the conformal weight of $X^i$. 
Assume that the conformal weight grading satisfies
\[
V = \bigoplus_{n \in \frac{1}{d} \ZZ_{\geq 0}} V_n
\]
for some positive integer $d$  with  $V_0 = \CC \vak$. 
Let 
\[
X^i(z)  := Y(X^i, z) =  \sum_{n \in \mathbb Z} X_n^i z^{-n-1}.
\]
Then one introduces the subspace $F_pV$ spanned by vectors of the form
\[
X^{i_1}_{-n_1} \dots X^{i_r}_{-n_r} \vak
\]
for $r\geq 0$ and $n_j \geq 1$ and satisfying $\Delta_{i_1} + \dots + \Delta_{i_r} \leq p$.
The subspaces satisfy
\begin{enumerate}
\item $F_p V \subset F_q V$ for $p <q$
\item $V = \bigcup_p F_pV$.
\item $X_n F_pV \subset  F_{p+q}V$ for $X \in F_pV$ and $n \in \ZZ$
 \item $X_n F_pV \subset  F_{p+q-1}V$ for $X \in F_pV$ and $n \in \ZZ_{\geq 0}$
\end{enumerate}
Note that the filtration does not depend on the choice of strong generating set and it is stable under $L_0$.
This filtration induces a filtration on the algebra of modes, $\mathcal A$,  \cite[Section 3.13]{Ara1}.

Let $M$ be a lower-bounded $V$-module, then the corresponding Li-filtration on $M$ is defined analogously, that is $F_p M$ is spanned by 
\[
X^{i_1}_{-n_1} \dots X^{i_r}_{-n_r}  m 
\]
for $r\geq 0$ and $n_j \geq 1$ and satisfying $\Delta_{i_1} + \dots + \Delta_{i_r} \leq p$ and $m$ a top level vector of $M$. It satisfies 
\begin{enumerate}
\item $F_p M \subset F_q M$ for $p <q$
\item $M = \bigcup_p F_pM$.
\item $X_n F_pM \subset  F_{p+q}M$ for $X \in F_pV$ and $n \in \ZZ$
 \item $X_n F_pM \subset  F_{p+q-1}M$ for $X \in F_pV$ and $n \in \ZZ_{\geq 0}$
\end{enumerate}
The last statement follows since if $X_n, Y_m$ have conformal weights $\Delta_X, \Delta_Y$, then $[X_n, Y_m]$ is a linear combination of elements in $\mathcal A$ of conformal weight at most $\Delta_X + \Delta_Y -1$. 

\begin{example}\label{ex:filtration}
Let $W= \vir^{\ell} \otimes \cF \otimes \pi^X$ then the Virasoro field $L$ has degree two, the Heisenberg field degree one and the two fermions $b$ and $c$ have degree $1/2$.  We will denote this filtration in the following by $F$.

Let $M$ be a $W$-module and let $G$ be Li's standard filtration on $M$ viewed as an  $\W^k(\ssl)$-module with an unusual choice of conformal structure.
In particular $L$ has degree two, $G^+ = b$ has degree $1/2$ and $G^-$ has degree $5/2$ and $J$ has degree $1$.
This is the filtration that is compatible with $F$ in the sense that these generators have the same degree with respect to $F$. In particular this means that 
\[
G_p(M) \subset F_p(M).
\]

Another filtration that we will use is the one where we exchange the degrees of $G^+$ and $G^-$, i.e. we give $G^+$ degree $5/2$ and $G^-$ degree $1/2$. This filtration will be called $\tilde G$. 
\end{example}

\subsection{$C_1$-cofiniteness in $\ssWtsl{k}$}

As before let $\ell$ and $k$ be related via $(\ell+2)(k+1) =1$ and let them be non-critical, i.e. $\ell \neq -2$ and $k \neq  -1$.

\begin{theorem}\label{keythm}
Let $V(c_\ell, h)$ be a Verma module for $\vir^\ell$ and let $L(c_\ell, h)$ be its simple quotient. Let $\ell \notin \mathbb Z_{\geq 0}$. 
$L(c_\ell, h)$ is $C_1$-cofinite as a module for $\vir ^{\ell}$ if and only if for all $\lambda \in \mathbb C$ the module
$Q = L(c_\ell, h) \otimes \cF \otimes \pi^X_\lambda$ is $C_1$-cofinite as a $\W^k(\ssl)$-module. 
\end{theorem} 
\begin{proof} 
Let $W= \vir^\ell \otimes \cF \otimes \pi^X$.
One direction is obvious: If $L(c_\ell, h) \otimes F \otimes \pi^X_\lambda$ is $C_1$-cofinite as a $\W^k(\ssl)$-module, then it is in particular $C_1$-cofinite as as a $W$-module. This can only be true if every tensor factor is $C_1$-cofinite.

We have to prove the converse.  
Set $M = V(c_\ell, h)\otimes F \otimes \pi^X_\lambda$ and let $\pi : M \rightarrow Q$ be the projection onto $Q$. 
Let $w$ be the top level vector of $M(c_\ell, h)$  and $z:=w \otimes | 0 \rangle  \otimes | \lambda \rangle$ be the one of $M$.
Here  $| 0 \rangle$ and   $| \lambda \rangle$ are the vacuum vector of $F$ and the highest-weight vector of $\pi^X_\lambda$.
 \\[-3mm]

\noindent {\bf Case 1:} We consider $\lambda$ generic in the sense of Remark \ref{rem:generic}, that is $\lambda$ is such that 
\[
 h \notin \left\{ \frac{1}{\ell+2}  \left( \left( r-1+ \frac{\ell}{4} \widetilde \lambda \right)^2 - (\ell+1)\left(r-1 + \frac{\ell}{4} \widetilde \lambda\right) \right) \ \Big|   \ r \in \mathbb Z_{>0} \ ,  \  \widetilde \lambda \in \{ \lambda, \lambda^\tau\} \ \right\}
\] 
Then 
\[
\mathcal V_{\Delta, \mu} \cong V(c_\ell, h)  \otimes \cF \otimes \pi^X_\lambda\qquad\text{and}\qquad
\mathcal V^\tau_{\Delta, \mu} \cong V(c_\ell, h^\tau)  \otimes \cF \otimes \pi^X_{\lambda^\tau}
\]
with $\Delta = \Delta(h, \lambda) = h - \frac{1}{2} - \frac{\ell^2 \lambda^2 }{16(\ell+2)} - \frac{\ell \lambda}{4}$ and $\mu =  \mu(\lambda) = -\frac{\ell \lambda}{2 (\ell+2)} -1$.
 \\[-3mm]

By assumption $L(c_\ell, h)$ is $C_1$-cofinite as a $\vir^{c_\ell}$-module and so $V(c_\ell, h)$ has a singular vector $S^{Vir}$ of the form $S^{Vir} = L_{-1}^N w + C_1(V(c_\ell, h))$ (Proposition \ref{sing}).
 \\[-3mm]
 
\noindent {\bf Step 1:} $T_{-1}^NG^+_{-1/2}  \pi(z)   \ \ \in \ C_1(Q)$.

Set $S = S^{Vir}  \otimes | 0 \rangle  \otimes | \lambda \rangle$, it is a singular vector of $M$ and so in particular its image in $Q$ is zero.  
$S$ has conformal weight $N+ h$ and $X_0$ weight $\lambda$ and note that the subspace of $G_{2N-1}(M)$ at weight $(N+h, \lambda)$ is clearly contained in $C_1(M)$. 
%The next steps are:
%\begin{enumerate}
%\item[(Step 1)] 
%We first determine the form of $S \mod G_{2N-1}(M)$, see \eqref{S mod G}. This leaves us with three coefficients $\alpha, \beta, \gamma$. We want to show that $\alpha$ and $\gamma$ are non-zero.
%\item[(Step 2)] We use the action of $G^+_{1/2}$ to show that either both $\alpha$ and $\gamma $ are zero or none of them.
%\item[(Step 3)] We repeat Step 1 and Step 2 for the filtration $\tilde G$. 
%\item [(Step 4)] Using the action of the $G^\pm_{1/2}$ we see that $G_{R}(Q) \subset C_1(Q)$ for $R >2N$ and in particular $Q/C_1(Q)$ is finite dimensional.
% \item [(Step 5)] The cases $2h_M+\lambda = 0$ and $2h_M - 2 -\lambda=0$ follow using Theorem \ref{miyamoto}.
%\end{enumerate}

We use Li's filtrations $F, G$ from Example \ref{ex:filtration}.
$W$ is strongly and freely generated by $L, b, c, X$.
Let $T$ be the Virasoro field of $\W^k(\ssl)$, then another such generating set is $T, b, c, J$. 
Denote by $F$ and $F'$ Li's standard filtrations with respect to these two sets of strong generators, they coincide $F = F'$. 

Recall Proposition \ref{sing}, saying that 
\[
S^{Vir}  = L_{-1}^Nw  + \dots \ .
\]
 Here the dots denote a sum of ordered monomials in the $L_{-1}, L_{-2}, \dots$ with at least one factor not equal to $L_{-1}$. In particular $L_{-1}^N$ is the only summand that has $L$-degree $N$ and all others have strictly smaller degree. Since $L_{-n}$ in $F_2$ it follows that 
the image of $S$ in $M/F_{2N-1}(M)$ is just  $ L_{-1}^Nz$.

Clearly the subspace of conformal weight $h_M+N$ and $J_0$ weight $\nu = - \frac{ \ell\lambda}{2(\ell+2)}$ is spanned by this vector 
\[
\left(M/F_{2N-1}(M)\right)_{ h + N, \nu} = \text{span}(L_{-1}^Nz). 
\]
We change perspective and replace the strong generator $L$ by $T$. Then this subspace is equivalently spanned by $T_{-1}^Nz$,
\[
\left(M/F_{2N-1}(M)\right)_{h+ N, \nu} = \text{span}(T_{-1}^Nz). 
\]

Let $G$ be Li's standard filtration on $M$ viewed as an $\W^k(\ssl)$-module with an unusual choice of conformal structure.
In particular $T$ has degree two, $G^+ = \psi^+$ has degree $1/2$ and $G^-$ has degree $5/2$ and $J$ has degree one.
This is the filtration that is compatible with $F$ in the sense that these generators have the same degree with respect to $F$, see Example \ref{ex:filtration}. In particular this means that 
\[
G_p(M) \subset F_p(M).
\]
This means in particular that $S \notin G_{2N-1}(M)$. We determine $S \mod G_{2N-1}(M)$. For this we retain Notation \ref{notation}
%For this introduce some notation 
%\begin{itemize}
%\item Let $A = (a_1, \dots, a_n) \in \mathbb Z^n$ with $a_1\geq a_2 \geq \dots \geq a_n \geq 1$ and set $|A| = n$. 
%\item Let $C = (c_1, \dots, c_m)\in \mathbb Z^m$ with $c_1\geq c_2 \geq \dots \geq c_m \geq 1$ and set $|C| = m$.
%\item Let $B = (b_1, \dots, b_r)\in \left(\frac{1}{2}\mathbb Z\right)^r$ with $b_1 >  b_2 > \dots > b_r \geq 1/2$ and $b_i \not\in \mathbb Z$ and set $|B| = r$.  
%\item Let $\tilde B = (\tilde b_1, \dots, \tilde b_s)\in \left(\frac{1}{2}\mathbb Z\right)^s$ with $\tilde b_1 >  \tilde b_2 > \dots >\tilde b_s \geq 1/2$ and  $\tilde b_i \not\in \mathbb Z$ set $|\tilde B| = s$.  
%\item Set $T_{-A} = T_{-a_1} \dots T_{-a_n}, G^+_{-B} = G^+_{-b_1} \dots  G^+_{-b_r},  G^-_{-\tilde B} = G^-_{-\tilde b_1} \dots  G^-_{-\tilde b_s}, X_{-C} = X_{-c_1} \dots  X_{-c_m}$. 
%\end{itemize}
and consider
\[
Y = X_{-C} T_{-A}G^+_{-B} G^-_{-\tilde B} z.
\]
%We want to understand under what conditions the image of $Y$ under the projection $M \rightarrow Q$ is in the $C_1$-subspace. For this we need to do some combinatorics. 

$Y$ having $J_0$ weight $\nu$ implies that $r=s$, i.e. $|B| = |\tilde B|$. 
Then $Y \notin G_{2N-1}(M)$ implies that $2|A| + 3|B| + |C| \geq 2N$
and the conformal weight of $Y$ being $h_M + N$ implies that $N \geq |A| + |C| + (1  + 3 + \dots  + (2|B|-1)) = |A|+|C| +|B|^2$.
We thus get
\[
2|A| + 3|B| + |C| \geq 2N \geq 2|A|+2|C| + 2|B|^2.
\]
The only possibilities are thus $(|A| =N, |B| = 0 = |C|)$, $(|A| = N-1, |B| = 1, |C| = 0)$ and $(|A| = N-2, |B| = 1 = |C|)$.
The unique corresponding monomials of conformal weight $h+N$ are 
\[
T_{-1}^Nz, \quad T_{-1}^{N-1}G^+_{-1/2} G^-_{-1/2}z \quad\text{and} \quad  X_{-1} T_{-1}^{N-2}G^+_{-1/2} G^-_{-1/2} z.
\]
These three vectors span 
\[
\left(G_{2N+1}(M)/G_{2N-1}(M)\right)_{h+N, \nu} = \text{span}(T_{-1}^Nz,  T_{-1}^{N-1}G^+_{-1/2} G^-_{-1/2} z,  X_{-1} T_{-1}^{N-2}G^+_{-1/2} G^-_{-1/2}z).
\]
It follows that 
\begin{equation} \label{S mod G} \nonumber
S \mod G_{2N-1}(M) = \alpha T_{-1}^Nz + \beta X_{-1} T_{-1}^{N-2}G^+_{-1/2} G^-_{-1/2}z + \gamma T_{-1}^{N-1}G^+_{-1/2} G^-_{-1/2} z 
\end{equation}
for certain $\alpha, \beta, \gamma$ at least one of them non-zero. From the proof of Corollary \ref{cor:mod_realized} we know that 
\[
G^{-}_{-1/2} z = p_1(h, \lambda) \psi^-_{-1/2} z
\]
with $p_1$ defined in \eqref{eq:pr}.
 Hence $T_{-1}^{N-1}G^+_{-1/2} G^-_{-1/2}z$  and $X_{-1} T_{-1}^{N-2}G^+_{-1/2} G^-_{-1/2} z$ are both in $F_{2N-1}(M)$, i.e. since $S \mod F_{2N-1}(M) = T_{-1}^Nz \mod F_{2N-1}(M)$ we have that $\alpha=1$. 

Clearly the weight $(N+h, \nu)$ subspace of $G_{2N-1}(M)$ is in $C_1(M)$ and so we have
\begin{equation} \label{eq:c1-cond}
S \mod C_1(M) =  T_{-1}^Nz  + \gamma T_{-1}^{N-1}G^+_{-1/2} G^-_{-1/2} z  \mod C_1(M).
\end{equation}
Let $Q = L(c_\ell, h) \otimes \cF \otimes \pi^X_\lambda$ with $L(c_\ell, h)$ the simple quotient of $V(c_\ell, h)$ as before. let $\pi : M \rightarrow Q$ be the projection of $M$ onto $Q$. 
In particular the image of $S$ in $Q$ is zero, $\pi(S) = 0$. Hence 
\[
T_{-1}^N \pi(z)  + \gamma T_{-1}^{N-1}G^+_{-1/2} G^-_{-1/2} \pi(z) \ \ \in \ C_1(Q).
\]
Since $G^+_{-1/2}T_{-1}^{N-1}G^+_{-1/2} G^-_{-1/2} \pi(z) \  \in \ C_1(Q)$ it follows that also 
\[
T_{-1}^N G^+_{-1/2} \pi(z)   \ \ \in \ C_1(Q).
\]
\\[-10mm]
 
\noindent {\bf Step 2:} $ T_{-1}^N G^-_{-1/2} \pi(z)   \ \ \in \ C_1(Q)$.

Repeating Step 1 for $\lambda$ replaced by $\lambda^\tau$, $z$ replaced by $z^\tau := w \otimes | 0 \rangle \otimes | \lambda^\tau \rangle$, 
$M$ replaced by $M^\tau = V(c_\ell, h) \otimes \cF \otimes \pi^X_{\lambda^\tau}$ and $Q$ replaced by  $Q^\tau = L(c_\ell, h) \otimes \cF \otimes \pi^X_{\lambda^\tau}$
 shows that $ T_{-1}^NG^+_{-1/2} \pi(z^\tau)$ is in $C_1(Q^\tau)$. The involution $\tau$ clearly preserves $C_1$-cofiniteness, i.e. $\tau(C_1(Q^\tau)) = C_1(Q)$. In particular 
 \[
 \tau\left(T_{-1}^N G^+_{-1/2}  \pi(z^\tau)\right) = T_{-1}^N G^-_{-1/2}  \pi(z)   \ \ \in \ C_1(Q).
 \]
\\[-10mm]
 
\noindent {\bf Step 3:}  $Q$ is $C_1$-cofinite.

Acting with $G^\pm_{-1/2}$ on $T_{-1}^N G^\mp_{-1/2}  \pi(z)$ shows that $T_{-1}^{N+1} \pi(z)$ and $T_{-1}^N G^+_{-1/2} G^-_{-1/2}  \pi(z)$ are also in $C_1(Q)$.
It follows that $Q/C_1(Q)$ is spanned by 
\[
T_{-1}^n z, \qquad T_{-1}^m G^\pm_{-1/2}z, \qquad T_{-1}^m G^+_{-1/2}G^-_{-1/2}z
\]
with $0 \leq m < N$ and $0 \leq n < N+1$, i.e. $Q$ is $C_1$-cofinite.
\\[-3mm]

%Now, we can interchange the role of $G^+$ and $G^-$ and in particular rewrite \eqref{S mod G} as 
%\begin{equation} \label{S mod tilde G}
%S \mod G_{2N-1}(M) = \tilde \alpha T_{-1}^Nz + \tilde \beta X_{-1} T_{-1}^{N-2}G^-_{-1/2} G^+_{-1/2}z + \tilde\gamma T_{-1}^{N-1}G^-_{-1/2} G^+_{-1/2} z. 
%\end{equation}
%Then the same reasoning gives 
%\[
%S \mod C_1(M) = \tilde \alpha T_{-1}^Nz  + \tilde \gamma T_{-1}^{N-1}G^+_{-1/2} G^-_{-1/2} z
%\]
%for non-zero $\tilde \alpha, \tilde \gamma$. 
%Let $L$ be the quotient of $M$ by the ideal generated by $S$. We denote the image of $z$ in $L$ by $z$ as well. 
%Since $G^+_{-1/2}T_{-1}^{N-1}G^+_{-1/2} G^-_{-1/2} z \in C_1(M)$ and $G^-_{-1/2}T_{-1}^{N-1}G^-_{-1/2} G^+_{-1/2} z \in C_1(M)$ as well it follows that 
%$T_{-1}^NG^\pm_{-1/2}z$ in $C_1(L)$. %
%Acting with $G^+_{-1/2}$ on $T_{-1}^NG^-_{-1/2}z$ we get that  $T_{-1}^NG^+_{-1/2}G^-_{-1/2}z$ in $C_1(L)$ and then also $T^{N+1}_{-1}z$ in $C_1(L)$ by 
%\eqref{eq:c1-cond}.  It follows that $L/C_1(L)$ is spanned by 
%\[
%T_{-1}^n z, \qquad T_{-1}^m G^\pm_{-1/2}z, \qquad T_{-1}^m G^+_{-1/2}G^-_{-1/2}z
%\]
%for $0 \leq m \leq N$ and $0 \leq n \leq N+1$, in particular the dimension is bounded by $4N+5$.

%We thus have proven our Theorem for all modules, except for the  cases $2h_M+\lambda = 0$ and $2h_M - 2 -\lambda = 0$. \\[-3mm]

\noindent{\bf Case 2:} $\lambda$ non-generic.

Recall that case 1 proves the Theorem for all but countable many $\lambda$. Let now $\lambda$ be non-generic, then we 
choose $\lambda', \lambda''$, s.t $\lambda'+\lambda'' =\lambda$ and such that both $\lambda', \lambda''$ are generic. 
The surjective intertwining operator ${ \pi^X_\lambda \choose \pi^X_{\lambda'} \, \pi^X_{\lambda''}}$ of Fock modules lifts to a surjective $\W^k(\ssl)$ intertwining operator 
\[
{ L(c_\ell, h) \otimes F \otimes \pi^X_\lambda \choose  \vir^{\ell} \otimes F \otimes \pi^X_{\lambda'} \quad  L(c_\ell, h) \otimes F \otimes \pi^X_{\lambda''} }
\]
and hence $L(c_\ell, h) \otimes F \otimes \pi^X_\lambda$ is $C_1$-cofinite as well by Theorem \ref{miyamoto}.
\end{proof}

\begin{corollary}\label{cor:C1}
Let $\ell$ be an admissible level for $\sltwo$ and $k$ be defined via $(\ell+2)(k+1) =1$. Then every object of  $\ssWtsl{k}$ is $C_1$-cofinite.
\end{corollary}
\begin{proof}
If $\ell$ is admissible and integral then $\W_k(\ssl)$ is strongly rational as a Heisenberg coset of a strongly rational vertex operator algebra \cite{CKLR}  and  in particular every object is even $C_2$-cofinite (which implies $C_1$-cofiniteness). 
Let $\ell$ be admissible and not integral. The simple Virasoro algebra $\vir^\ell$ is strongly rational and in particular every of its modules is $C_1$-cofinite. Hence by the previous Theorem every $\vir^{\ell} \otimes \cF \otimes \pi^\gamma$-module is $C_1$-cofinite as a $\W_k(\ssl)$-module. 
Lemma \ref{c1cofinite} together with Proposition \ref{prop;emb} imply that every simple $\W_k(\ssl)$-module is $C_1$-cofinite and hence every object that is of finite length by Lemma \ref{c1cofinite} again. But $\ssWtsl{k}$ is of finite length by Corollary \ref{cor:fl}.
\end{proof}

\section{Existence of vertex tensor category}\label{tc}

The existence of vertex tensor categories is in general a difficult problem. The best criterion for us is 
\begin{theorem}\label{thm:existence}
Let $V$ be a vertex operator superalgebra. If every lower-bounded $C_1$-cofinite $V$-module is of finite-length and if the category of lower-bounded $C_1$-cofinite modules is closed under the contragredient dual functor, then the category of  lower-bounded $C_1$-cofinite $V$-modules is a vertex tensor supercategory. 
\end{theorem}
This Theorem arose from \cite[Section 4]{C-Vir}, has been proven for vertex operator algebras \cite[Section 3]{CY} and phrased much better in \cite[Theorem 2.3]{Mc-mirror}. Finally \cite[section 2.3]{CMY1} notes that this works equally well for vertex operator superalgebras. 

Let $V$ be a vertex operator superalgebra and $A \supset V$ a conformal extension of $V$. Let $\mathcal C$ be a vertex tensor category of $V$ modules and assume that $A$ is an object in $\mathcal C$. Then $A$ is identified with a commutative superalgebra in $\mathcal C$ \cite{HKL, CKL}. 
A commutative superalgebra $A$ in a braided tensor supercategory $\mathcal C$ gives rise to a braided tensor supercategory $\text{Rep}^0(A)$ of local $A$-modules in $\mathcal C$. Moreover $\text{Rep}^0(A)$ is isomorphic as a braided tensor supercategory to the vertex tensor supercategory of modules of the vertex operator superalgebra $A$ that lie in $\mathcal C$ \cite{CKM}.

Finally, there is a tensor functor $\cF$ from $\mathcal C$ to $\text{Rep}(A)$, the tensor category of  not necessarily local $A$-modules. One has that $\cF(X)$ is in $\text{Rep}^0(A)$ if and only if the monodromy (double braiding) of $X$ with $A$ is the identity on $X \boxtimes A$, $\mathcal M_{A, X} = \text{Id}_{A \boxtimes X}$. Here $\boxtimes$ denotes the tensor product in $\mathcal C$.
As a $\mathcal C$-module $\mathcal F(X) \cong_{\mathcal C} A \boxtimes X$.

Sometimes an extension $A$ is not of finite index, i.e. an infinite sum of $V$-modules. In this instance the above statements are generalized to the direct limit completion $\text{Ind}(\mathcal C)$
of $\mathcal C$ in \cite{CMY2}.

Our setting is of the form (compare with  \eqref{eq:ext})
\[
A = \bigoplus_{n \in \mathbb Z} V_n \otimes \pi_n
\]
with $V$ some vertex operator algebra, the $V_n$ simple $V$-modules that lie in a vertex tensor category of $V$-modules and $\pi$ a rank one Heisenberg vertex operator algebra of some non-zero level and $\pi_n$ Fock modules of weight $n$. 

This setting has appeared a few times: 
\begin{enumerate}
\item 
In \cite{CMY3}
in the case that $V$ is the singlet vertex algebra $\mathcal M(p)$ and $A$ the $\mathcal B(p)$-algebra for $p\in \mathbb Z_{\geq 2}$. 
\item In \cite{CMY1} in the case that $V$ is $L_k(\mathfrak{gl}_{1|1})$, the affine  vertex superalgebra of $\mathfrak{gl}_{1|1}$ at any non-zero level $k$ and $A$ is $L_{-1/2}(\mathfrak{sl}_{2|1})$ the simple affine  vertex  superalgebra of $\mathfrak{sl}_{2|1}$. at level $-1/2$. 
\item In \cite{ACPV} in the case that $V$ is $L_{-n-\frac{1}{2}}(\mathfrak{sl}_{2n}) \otimes \mathcal M(2)$ and $W$ is. the minimal $W$-algebra of $\mathfrak{sl}_{2n+2}$ at level $-n -\frac{3}{2}$. 
\end{enumerate}
In particular in the first two cases it has been used to get vertex tensor (super)category structure on a category of non lower-bounded (in particular non $C_1$-cofinite) modules. 

We now have a similar result with $V = \W_k(\ssl)$ and $A  =   L_\ell(\sltwo) \otimes \cF$.
\begin{theorem}\label{thm:main}
Let $\ell$ be an admissible level for $\sltwo$ and $k$ be defined via $(\ell+2)(k+1) =1$. Then $\ssWtsl{k}$  is a vertex tensor supercategory and $\sWtsl{\ell}$ is a vertex tensor category in the sense of \textup{\cite{HLZ1, HLZ2, HLZ3, HLZ4, HLZ5, HLZ6, HLZ7, HLZ8}}.
\end{theorem}
\begin{proof}
The first statement follows from Theorem \ref{thm:existence} together with Corollary \ref{cor:C1} and Corollary \ref{cor:fl}.
The second statement follows, namely from \eqref{eq:ext}
 $A:= L_\ell(\sltwo) \otimes \cF$ is a vertex operator algebra extension of $\W_k(\ssl) \otimes \pi$ for $\pi$ a rank one Heisenberg vertex operator algebra. The category $\mathcal{H}$ of Fock modules of $\pi$ 
  has a vertex tensor category structure, first stated in \cite{CKLR}.
 
 The tensor product of the two vertex operator algebras has a vertex tensor category structure as well and this is the Deligne product $\mathcal C:= \ssWtsl{k} \boxtimes \mathcal H$ of the categories \cite{CKM2, Mc-deligne}. Its direct limit completion $\text{Ind}(\mathcal C)$ is also a vertex tensor category \cite[Theorem 1.2]{CMY2} and then the category $\text{Rep}^0(A)$ of $L_\ell(\sltwo) \otimes \cF$ -modules that lie in this direct limit completion as well \cite[Theorem 1.4]{CMY2}. 
 The category of modules of $\cF$ is just vector superspaces and so this category of local modules is just $\sWtsl{\ell} \boxtimes \, \text{sVect}$ and
 in particular $\sWtsl{\ell}$ is a vertex tensor category.
\end{proof}

\subsection{Translating ribbon categories}

Let $k = -2 + \frac{u}{v}$ be an admissible level as before. 
We now use the theory of vertex algebra extensions to study the category $\sWtslA{k+1}$. 
Recall the coset decomposition
\[
A := \slirr{1}^k \otimes \slirr{1}^1  \cong \bigoplus_{\substack{ m=1 \\ m \ \text{odd} }}^{u+v-1} \slirr{m}^{k+1} \otimes \vir_{m, 1}^{k'}, \qquad k' = \frac{k+3}{k+2} -2. 
\]
In particular for  $\mathcal C = \sWtsl{k+1} \boxtimes \vir_{k'}\on{mod}$  the object $A$ is a commutative algebra in $\mathcal C$ and  $\sWtsl{k} \boxtimes\sWtsl{1} \cong
\text{Rep}^0(A)$.
For $r =1, \dots, u-1$ we define functors 
\[
\mathcal F_r : \sWtslA{k+1} \rightarrow \text{Rep}(A)
\]
by $\mathcal F_r(X) = \mathcal F(X \otimes \vir^{k'}_{1, r})$ on objects and $\mathcal F_r(f) = \mathcal F(f \otimes \text{Id}_{\vir_{1, r}^{k;}})$ on morphisms.

Note that the only identification of $\vir_{k'}$-modules are $\vir^{k'}_{m, r} \cong \vir_{u+v-m, u-r}$. 
In particular
$\text{Hom}(\vir^{k'}_{1, r} , \vir^{k'}_{m, r} ) = \mathbb C$ for $m=1$ and if $u$ is even also for $m = u+v-1$ and $r =\frac{u}{2}$. Otherwise it vanishes. If $u$ is even then $v$ is odd and so $u+v$ is odd and so $u+v-1$ is even in particular 
\begin{equation}\label{Hom-Vir}
\text{Hom}(\vir^{k'}_{1, r} , \vir^{k'}_{m, r} ) = \mathbb C \delta_{1, m} \qquad \text{if} \ m \ \text{is odd}.
\end{equation}

\begin{proposition}\label{prop:Fr}
Each $\mathcal F_r$ is fully faithful with image in  $\sWtsl{k} \boxtimes\sWtsl{1}$.
\end{proposition}
\begin{proof}
Consider two objects $X, Y$ in $\sWtslA{k+1}$.  Then
\begin{equation}\label{eq:FR}
\begin{split}
\text{Hom}_{\text{Rep}(A)}&\left( \mathcal F_r(X), \mathcal F_r(Y)\right) = \text{Hom}_{\text{Rep}(A)}\left( \mathcal F(X \otimes  \vir^{k'}_{1, r}), \mathcal F(Y \otimes \vir^{k'}_{1, r})\right) \\
&= \text{Hom}_{\mathcal C}\left(X \otimes  \vir^{k'}_{1, r}, \bigoplus_{\substack{ m=1 \\ m \ \text{odd} }}^{u+v-1} (\slirr{m}^{k+1} \boxtimes Y) \otimes \vir^{k'}_{m, r})\right) \\
&= \text{Hom}_{\mathcal C}\left(X \otimes  \vir^{k'}_{1, r}, (\slirr{1}^{k+1} \boxtimes Y) \otimes \vir^{k'}_{1, r}\right) \\
&= \text{Hom}_{\sWtsl{k} }\left(X, Y\right).
\end{split}
\end{equation}
Here the second line is Frobenius reciprocity \cite[Lemma 2.61]{CKM} and  the third one uses  \eqref{Hom-Vir}.
So each $\mathcal F_r$ is fully faithfull. 
It remains to show that its image is local, i.e. in $\sWtsl{k} \boxtimes\sWtsl{1}$.
Let $1 \leq r \leq u-1, 1 \leq s \leq v-1$ and let
$a$ denote an integer, s.t. $a = r+s \mod 2$.
Using the fusion rules, Theorem \ref{thm:fusion} and \eqref{eq:fus-vir}, as well as Corollary \ref{cor:branching}
\begin{equation}\label{eq:ident}
\begin{split}
\mathcal F_r(\sigma^\ell(\slrel{\lambda + \frac{a+1}{2}}{1, s}^{k+1})) &\cong_{\mathcal C} \bigoplus_{\substack{m =  1 \\  m +r + s + a \ \text{odd}  }}^{u+v-1} \sigma^\ell(\slrel{\lambda + \frac{a+1}{2}}{m, s}^{k+1}) \otimes \vir_{m, r}^{k'}  \cong_{\mathcal C}
\sigma^\ell(\slrel{\lambda}{r,s}^k) \otimes \slirr{\underline{a+\ell}}^1 \\ 
%\mathcal F_r(\sfmod{\ell}{\slindrel{m, s}^{\pm, k+1}} ) &\cong_{\mathcal C} \bigoplus_{\substack{m =  1 \\  m +r + s + a \ \text{odd}  }}^{u+v-1} \sfmod{\ell}{\slindrel{m, s}^{\pm, k+1}} \otimes \vir_{m, r}^{k'}  \cong_{\mathcal C}
% \sfmod{\ell}{\slindrel{r,s}^{\pm, k}}\otimes \slirr{\underline{a+\ell}}^1 \\  
\mathcal F_r(\sfmod{\ell}{\sldis{1, s}^{\pm, k+1}} )&\cong_{\mathcal C} \bigoplus_{\substack{m =  1 \\  m +r + s + a \ \text{odd}  }}^{u+v-1} \sfmod{\ell}{\sldis{m, s}^{\pm, k+1}} \otimes \vir_{m, r}^{k'}  \cong_{\mathcal C} 
 \sfmod{\ell}{\sldis{r,s}^{\pm, k}}\otimes \slirr{\underline{a+\ell}}^1  \\  
 \mathcal F_r( \sfmod{\ell}{\slproj{1, s}^{k+1}})&\cong_{\mathcal C}   \bigoplus_{\substack{m =  1 \\  m +r + s + a \ \text{odd}  }}^{u+v-1} \sfmod{\ell}{\slproj{m, s}^{k+1}} \otimes \vir_{m, r}^{k'} \cong_{\mathcal C}  
  \sfmod{\ell}{\slproj{r,s}^k}\otimes \slirr{\underline{a+\ell}}^1 \end{split}
\end{equation}
for $\lambda \neq \lambda_{r, s}, \lambda_{u-r, v-s}$.
Consider a pair of objects $(X, Z)$ 
in\\  $\left\{  (\sigma^\ell(\slrel{\lambda + \frac{a+1}{2}}{1, s}^{k+1}), \sigma^\ell(\slrel{\lambda}{r,s}^k) \otimes \slirr{\underline{a+\ell}}^1), 
(\sfmod{\ell}{\sldis{1, s}^{\pm, k+1}}, \sfmod{\ell}{\sldis{r,s}^{\pm, k}}\otimes \slirr{\underline{a+\ell}}^1 ), 
( \sfmod{\ell}{\slproj{1, s}^{k+1}}, \sfmod{\ell}{\slproj{r,s}^k}\otimes \slirr{\underline{a+\ell}}^1 ) \right\}$.
By Frobenius reciprocity \cite[Lemma 2.61]{CKM} and  \eqref{Hom-Vir}, that is by the same reasoning as for \eqref{eq:FR}
\begin{equation}\label{eq:FR2}
\begin{split}
\text{Hom}_{\text{Rep}(A)}\left( \mathcal F_r(X), Z\right) &\cong \text{Hom}_{\text{Rep}(A)}\left( \mathcal F(X \otimes  \vir^{k'}_{1, r}), Z \right) \\
&\cong \text{Hom}_{\mathcal C}\left(X \otimes  \vir^{k'}_{1, r}, \bigoplus_{\substack{ m=1 \\ m \ \text{odd} }}^{u+v-1} (\slirr{m}^{k+1} \boxtimes X) \otimes \vir^{k'}_{m, r})\right) \\
&\cong \text{Hom}_{\mathcal C}\left(X \otimes  \vir^{k'}_{1, r}, (\slirr{1}^{k+1} \boxtimes X) \otimes \vir^{k'}_{1, r}\right) \\
&\cong \text{Hom}_{\sWtsl{k} }\left(X, X\right).
\end{split}
\end{equation}
Here the second line also used \eqref{eq:ident}. In particular for those $Z$ that are simple we can already conclude that \eqref{eq:ident} are isomorphisms in $\text{Rep}(A)$, that is
\begin{equation}\nonumber
\begin{split}
\mathcal F_r(\sigma^\ell(\slrel{\lambda + \frac{a+1}{2}}{1, s}^{k+1})) &\cong_{\text{Rep}(A)}
\sigma^\ell(\slrel{\lambda}{r,s}^k) \otimes \slirr{\underline{a+\ell}}^1 \\ 
\mathcal F_r(\sfmod{\ell}{\sldis{1, s}^{\pm, k+1}} )&\cong_{\text{Rep}(A)}  \sfmod{\ell}{\sldis{r,s}^{\pm, k}}\otimes \slirr{\underline{a+\ell}}^1 
 \end{split}
\end{equation}
for $\lambda \neq \lambda_{r, s}, \lambda_{u-r, v-s}$.
Let now $(X, Z) =
( \sfmod{\ell}{\slproj{1, s}^{k+1}}, \sfmod{\ell}{\slproj{r,s}^k}\otimes \slirr{\underline{a+\ell}}^1 )$.
Consider the image $f$ of $\text{Id}_X$  in  $\text{Hom}_{\text{Rep}(A)}\left( \mathcal F_r(X), Z\right)$ 
under \eqref{eq:FR2}. $f$ is given in Lemma 2.61 of \cite{CKM} and in particular the image of $f$ must contain $X \otimes \vir^{k'}_{1, r}$ itself as a $\mathcal C$-submodule. 
It follows that the quotient of $Z$ by the image of $f$ is a local module that doesn't contain $\vir^{k'}_{1, r}$ as a direct summand, but clearly $Z$ doesn't have such a quotient and so 
$f$ must be surjective. But $f$ is in particular a morphism in $\mathcal C$ and since $\mathcal F_r(X)$ and $Z$ are isomorphic as objects in $\mathcal C$ by \eqref{eq:ident} it must already be an isomorphism
 in $\mathcal C$ and so it must be one in $\text{Rep}(A)$ as well. It follows that $\mathcal F_r(P)$ is in $\sWtsl{k} \boxtimes\sWtsl{1} \cong
\text{Rep}^0(A)$ for any projective module in $\sWtslA{k+1}$.  $\mathcal F_r$ is right exact (in fact it is exact, since $A$ is rigid, but we don't need this here) and so $\mathcal F_r(M)$ is a quotient of $\mathcal F_r(P)$ for $P$ the projective cover of $M$ and so $\mathcal F_r(M)$ is in $\sWtsl{k} \boxtimes\sWtsl{1} \cong
\text{Rep}^0(A)$ for any object $M$ in $\sWtslA{k+1}$.
\end{proof}
A module $M$ lifts to a local module $\mathcal F(M)$ if and only if the monodromy (double braiding) of $M$ with $A$ is trivial, that is $\mathcal M_{A, M} = \text{Id}_{A \boxtimes M}$ \cite[Proposition 2.65]{CKM}. We thus proved that $\sfmod{\ell}{\slproj{1, s}^{k+1}} \otimes \vir^{k'}_{1. r}$ has trivial monodromy with $A$. In particular this monodromy is semisimple and so the same must be true for the monodromy $\mathcal M_{\sfmod{\ell}{\slproj{1, s}^{k+1}}, \slirr{r}^{k+1}}$.
Corollary 2.68 of \cite{CKM} tells us that up to a natural isomorphism monodromy commutes with the induction functor $\mathcal F$ if applied to modules that induce to local modules. 
\begin{corollary}
$\mathcal M_{\sfmod{\ell}{\slproj{r, s}^{k}}, \slirr{t}^{k}}$ is semisimple for all $\ell \in \mathbb Z, 1 \leq r, t \leq u-1, 1 \leq s \leq v-1$. 
\end{corollary}

\begin{corollary}\label{cor:subcategory}
$\sWtslA{k+1}$ is a tensor subcategory of $\sWtsl{k+1}$.
\end{corollary}
\begin{proof}
We need to show that  $\sWtslA{k+1}$ is closed under tensor product. Consider two objects $X, Y$ in $\sWtslA{k+1}$. Since $\mathcal F_1$ is a composition of tensor functors it is tensor as well, i.e.
\[
\mathcal F_1(X) \boxtimes \mathcal F_1(Y) \cong \mathcal F_1(X \boxtimes Y). 
\]
In particular since both $\mathcal F_1(X)$ and $\mathcal F_1(Y)$ are in   $\sWtsl{k} \boxtimes\sWtsl{1} \cong
\text{Rep}^0(A)$, the same must be true for $ \mathcal F_1(X \boxtimes Y)$.  
Since
\[
\mathcal F_1(X \boxtimes Y) \cong_{\mathcal C} \bigoplus_{\substack{ m=1 \\ m \ \text{odd} }}^{u+v-1} (\slirr{m}^{k+1} \boxtimes (X \boxtimes Y)) \otimes \vir_{m, 1}^{k'} 
\]
in $\sWtsl{k} \boxtimes\sWtsl{1} \cong
\text{Rep}^0(A)$ and by inspecting Corollary \ref{cor:branching} every simple module of  $\text{Rep}^0(A)$ has the property that its multiplicity of $\vir^{k'}_{1, 1}$ is in $\sWtslA{k+1}$. 
Since $\sWtslA{k+1}$ is closed under extensions every module of  $\text{Rep}^0(A)$ has the property that it's multiplicity of $\vir^{k'}_{1, 1}$ is in $\sWtslA{k+1}$.  But $X \boxtimes  Y$ is precisely this multiplicity of $\mathcal F_1(X \boxtimes Y)$  by \eqref{Hom-Vir}.
\end{proof}
%Thus  $\mathcal F_1$ is a fully faithful tensor functor from $\sWtslA{k+1}$ to $\sWtsl{k} \boxtimes\sWtsl{1}.$
\begin{corollary}\label{cor;tensor}
$\sWtslA{k+1}$ is tensor equivalent to the image of $\mathcal F_1$ in   $\sWtsl{k} \boxtimes\sWtsl{1} $.
\end{corollary}
\begin{corollary}\label{cor:translation}
$\sWtsl{k+1}$ is a ribbon category  if and only if $\sWtsl{k}$ is a ribbon category.
\end{corollary}
\begin{proof}
First we note that the definiton of braiding in a vertex tensor category of a $\mathbb Z$-graded vertex operator algebra ensures that the twist $\theta = e^{2 \pi i L_0}$ balances the braiding, see \cite[Theorem 4.1]{H3}. We thus have to prove that rigidity on one category implies rigidity on the other one.

If $\sWtsl{k+1}$ is rigid, then $\mathcal C$ is rigid and then $\sWtsl{k} \boxtimes \sWtsl{1}$ is a ribbon category by Proposition 2.86 of \cite{CKM} and so is then $\sWtsl{k}$ as well as a tensor subcateogry of a ribbon category that is closed under duality.

Assume that $\sWtsl{k}$ is a ribbon category. 
The contragredient dual $M'$ of a lower-bounded module $M$ is given by the Weyl reflection of $\mathfrak{sl}_2$ corresponding to  the simple root and the contragredient dual of $\sigma^\ell(M)$ is $\sigma^{-\ell}(M')$ and in particular we inspect that 
$\sWtslA{k+1}$ is closed under  contragredient duals.  Note that $\vir^{k'}_{1, r}$ is its own contragredient dual.
Let $M$ in $\sWtslA{k+1}$, since $A$ is its own contragredient dual we have $\mathcal F_r(M') \cong_{\mathcal C} \mathcal F_r(M)'$. The exact same argument as in the proof of Proposition \ref{prop:Fr} gives $\mathcal F_r(M') \cong_{\text{Rep}(A)} \mathcal F_r(M)'$ and in particular the image under $\mathcal F_1$ is closed under the contragredient dual. 
In a ribbon vertex tensor category of a self-dual vertex operator algebra the contragredient dual is also the dual in the categorical sense (see section 4 of \cite{CMY4}), i.e. the image of $\mathcal F_1$ is closed under duality and it is clearly also abelian and so in particular it is a ribbon subcategory. By the previous corollary $\sWtslA{k+1}$ is thus a ribbon category as well. 
From the fusion rules, we see that every simple/projective object is the fusion product of a simple/projective object in $\sWtslA{k+1}$ with a $\slirr{m}^{k+1} $. 
Rigidity of the $\slirr{m}^{k+1}$ is shown in \cite{CHY} 
 and since the fusion product of rigid objects is rigid \cite[Lemma 4.3.5]{CMY4} we conclude that all simple  and projective objects in $\sWtsl{k+1}$ are rigid. 
 Also note that all projective objects are injective and vice versa \cite{ACK}. 
The assumptions of \cite[Prop.19]{TW} are satisfied and so $\sWtsl{k+1}$ is rigid. 
\end{proof}
In \cite{CMY3} the ribbon structure of $\sWtsl{k}$ for $k=-1/2$ and $k=-4/3$ have been established and also fusion rules are determined. 
\begin{corollary}\label{cor:ribbon}
$\sWtsl{k}$ is a ribbon category for admissible level $k = -2 + \frac{u}{v}$ and $v \in \{2, 3\}$ and $u = -1 \mod v$. Moreover the fusion rules follow from \textup{\cite{CMY3}} via Corollary \ref{cor;tensor}.
\end{corollary}

\flushleft
%\bibliography{RHWC1}
\bibliographystyle{unsrt}

\end{document}